\documentclass[reqno]{amsart}
\UseRawInputEncoding
\usepackage{a4wide,graphicx,tikz}
\usepackage{amsmath}
\usepackage{tikz}
\usepackage[T1]{fontenc}
\usepackage{lmodern}
\usepackage{mathtools}
\usepackage{amssymb}
\usepackage{microtype}
\usepackage{xcolor}
\usepackage[colorlinks=true,linkcolor=blue,citecolor=blue,urlcolor=blue]{hyperref}

\newif\ifshowrevisions
\showrevisionsfalse
\definecolor{revisionred}{RGB}{200,0,0}
\newcommand{\revisionformat}{%
  \ifshowrevisions
    \color{revisionred}%
    \hypersetup{linkcolor=revisionred,citecolor=revisionred,urlcolor=revisionred}%
  \fi
}
\usepackage{xcolor}
\newcommand{\rev}[1]{{#1}}

\newtheorem{theorem}{Theorem}[section]
\newtheorem{lemma}[theorem]{Lemma}
\newtheorem{proposition}[theorem]{Proposition}
\newtheorem{corollary}[theorem]{Corollary}

\newtheorem{question}{Question}
\theoremstyle{definition}
\newtheorem{definition}[theorem]{Definition}
\theoremstyle{remark}
\newtheorem{remark}[theorem]{Remark}
\newtheorem{example}[theorem]{Example}

\newcommand{\M}{\mathcal M}
\newcommand{\N}{\mathcal N}
\newcommand{\Core}{\mathcal C}
\newcommand{\E}{\mathbb E}
\newcommand{\PP}{\mathbb P}
\newcommand{\F}{\mathcal F}
\newcommand{\Trc}{\operatorname{Tr}}
\newcommand{\Succ}{\operatorname{Succ}}
\newcommand{\Term}{\operatorname{Term}}

\newcommand{\weakto}{\mathrel{\rightharpoonup}}
\newcommand{\one}{\mathbf 1}

\newcommand{\thmref}[1]{Theorem~\ref{#1}}
\newcommand{\lemref}[1]{Lemma~\ref{#1}}
\newcommand{\propref}[1]{Proposition~\ref{#1}}
\newcommand{\corref}[1]{Corollary~\ref{#1}}

\title[Martingale Transforms and Subsequence Estimates]{Martingale Transforms and Subsequence Estimates in Noncommutative \(L_p\)-Spaces}

\author[Jinghao Huang]{J. Huang}
\thanks{Y. Zhu is the corresponding author. This work was supported by the
  NNSF of China (Nos. 12301160, 12471134 and 12671159).}

\author[Yunpeng Zhu]{Y. Zhu}
\address{Institute for Advanced Study in Mathematics of HIT, Harbin Institute
  of Technology, Harbin 150001, China}
\email{jinghao.huang@hit.edu.cn}
\email{24s012030@stu.hit.edu.cn}


\begin{document}

\begin{abstract}
We establish a transference principle from martingale transform inequalities to subsequence estimates for weakly null sequences. 
As applications,  for every von Neumann algebra $\M$, 
  (1) we show that, 
 every bounded sequence
in the Haagerup space $L_1(\M)=\M_*$ has either a norm-convergent
subsequence or a $2$-co-lacunary subsequence,  answering an open problem
of Parcet and Randrianantoanina;
(2)
we prove that every UMD Banach space
has the uniform USP, which in particular yields uniform USP for
$L_p(\M)$, $1<p<\infty$, answering a 
question  posed by Oikhberg and Spinu; (3)
we answer an open problem of Haagerup, Rosenthal and Sukochev for the upper
$\ell_p$ estimates for $L_p(\M)$,  $1<p<2$, 
in the affirmative; (4)
 we also answer the problem of whether every $2$-strictly singular operator on a noncommutative $L_1$-space is automatically Dunford--Pettis. 
\end{abstract}

\maketitle

\section{Introduction}

A sequence $(x_n)$ in a Banach space $X$ is called
\emph{$2$-co-lacunary} if there is $\gamma>0$ such that
\[
  \left\|\sum_n a_nx_n\right\|_X
  \geq\gamma\left(\sum_n|a_n|^2\right)^{1/2}
\]
for every finitely supported scalar sequence $(a_n)$.

Aldous and Fremlin \cite[Theorem~6]{AldousFremlin1982}
proved that every bounded sequence in a classical
$L_1$-space has either a norm-convergent subsequence
or a $2$-co-lacunary subsequence.
R{\"a}biger \cite[Corollary~2.5, Examples~1.3(c)--(d)]
{RabigerOrderBounded1991} obtained corresponding results for
Lorentz spaces $L_{p,1}(\mu)$ on finite measure spaces,
$1\leq p<\infty$, and certain Orlicz function spaces on $[0,1]$.

In the noncommutative setting, Parcet and Randrianantoanina
proved that every relatively weakly compact,
seminormalized martingale difference sequence in
a noncommutative $L_1$-space on a semifinite von Neumann algebra
is $2$-co-lacunary
(see \cite[Theorem~3.6]{ParcetRandrianantoanina2006}).
They further obtained the following result: if $\M$ is a hyperfinite
semifinite von Neumann algebra, then every bounded sequence in
$L_1(\M,\tau)$ admits either a norm-convergent subsequence or a
$2$-co-lacunary subsequence
(see \cite[Corollary~3.7]{ParcetRandrianantoanina2006}).
Subsequently, Sukochev and Zhou
\cite[Theorem~1.1]{SukochevZhou2020}
established a corresponding subsequence characterization
for noncommutative symmetric Banach spaces associated
with hyperfinite semifinite von Neumann algebras.

However, it was unknown whether one can remove the
hyperfiniteness assumption, which is the main motivation of the present paper. 

\begin{question}\label{q:PR}\cite[Remark~3.11]{ParcetRandrianantoanina2006}
Does Corollary~3.7 remain valid without
the hyperfiniteness assumption?
\end{question}

We answer Question~\ref{q:PR} affirmatively and further
remove the semifiniteness assumption. For an
arbitrary von Neumann algebra $\M$, the Haagerup space
$L_1(\M)$ is canonically and isometrically identified
with the predual $\M_*$ \cite{Haagerup1979}.

Our proof begins with a result on martingale inequalities with respect to 
weakly null sequences in a Banach space. 
It applies to bounded linear operators from
Banach spaces into quasi-Banach spaces and gives a subsequence
estimate from an inequality for finite martingale differences.

\begin{theorem}\label{thm:transfer}
Let $X$ be a Banach space, let $Y$ be an $r$-normed quasi-Banach
space, $0<r\le1$, and let $J:X\to Y$ be bounded and linear.
Suppose that, for some $1\le s<\infty$ and $C>0$,
\begin{equation}\label{eq:martingale-input}
 \left(\E\left\|\sum_{k=1}^m\xi_{k-1}Jd_k\right\|_Y^r\right)^{1/r}
 \le C\left(\E\left\|\sum_{k=1}^m d_k\right\|_X^s\right)^{1/s}
\end{equation}
for every finite $X$-valued martingale difference sequence
$(d_k)_{k=1}^m$ and every predictable sign sequence
$(\xi_{k-1})_{k=1}^m$.
Then, for every $C'>C$, every seminormalized weakly null sequence
$(x_n)\subset X$ has a basic subsequence $(u_k)$ satisfying
\begin{equation}\label{eq:transfer-conclusion}
 \left\|\sum_k\gamma_ka_kJu_k\right\|_Y
 \le C'\left\|\sum_k a_ku_k\right\|_X
\end{equation}
for all finitely supported complex scalar sequences $(a_k)$ and
all sign sequences $(\gamma_k)$.
\end{theorem}


Let $\M$ be a von Neumann algebra equipped with
a normal faithful tracial state $\tau$.
For $0<r<1$, it follows from
\cite[Theorem~3.1]{ParcetRandrianantoanina2006}
that
\eqref{eq:martingale-input} holds with
$X=L_1(\M,\tau)$, $Y=L_r(\M,\tau)$,
 where $J$ is  the natural inclusion from $X$ into $Y$, and $s=1$.
Then, by Theorem~\ref{thm:transfer},
every seminormalized weakly null sequence in
$L_1(\M,\tau)$ has a basic subsequence $(u_k)$
satisfying
\begin{align}\label{eq:L_r}
 \left\|\sum_k\gamma_ka_ku_k\right\|_r
 \le A_r\left\|\sum_k a_ku_k\right\|_1
\end{align}
for all finitely supported scalar sequences $(a_k)$
and all sign sequences $(\gamma_k)$.

Combining the estimate \eqref{eq:L_r} with the
noncommutative \(L_p\)-cotype inequality
\cite[Corollary~5.8]{PisierXu2003} and the quasi-Banach
Kahane--Khintchine inequality
\cite[Proposition~2.1]{Kalton2005}, we conclude that,
whenever $\M$ admits a normal faithful finite trace $\tau$,
every seminormalized weakly null sequence in
$L_1(\M,\tau)$ has a $2$-co-lacunary subsequence (see Corollary \ref{cor:finite-weak-null} below).

We then extend this conclusion to arbitrary von Neumann algebras
and obtain the following theorem. 
\begin{theorem}\label{thm:main}
Let $\M$ be an arbitrary von Neumann algebra. Every bounded sequence in
the Haagerup space $L_1(\M)=\M_*$ has either a norm-convergent
subsequence or a $2$-co-lacunary subsequence.
\end{theorem}

We conclude the paper 
by answering several open problems as applications of 
 the above   two theorems.

In their paper \cite{OikhbergSpinu2014}, Oikhberg and Spinu studied the unconditional subsequence property (USP) of noncommutative $L_1$ associated with a semifinite von Neumann algebra, and asked the following question:
\begin{quote}
Suppose $\tau$ is a normal faithful
semifinite trace on a von Neumann algebra $\mathcal A$ and $1<p<\infty$. Does $L_p(\tau)$ have the USP?~\cite[Question~2.5]{OikhbergSpinu2014}
\end{quote}

A closely related problem due to Haagerup, Rosenthal and Sukochev \cite{HaagerupRosenthalSukochev2003} concerns the upper 
$\ell_p$ estimate of noncommutative $L_p$-spaces:
\begin{quote}
    Let \(1<p<2\) and \((f_n)\) be a seminormalized weakly null sequence
in \(L^p(\mathcal N)\) (\(\mathcal N\) a finite von Neumann algebra)
such that \((|f_n|^p)\) is not uniformly integrable.
Does \((f_n)\) have a subsequence equivalent to the usual
\(\ell^p\) basis? \cite[Section~5, Problem]{HaagerupRosenthalSukochev2003}
\end{quote}
In the latter discussion in \cite[Section~5, Problem]{HaagerupRosenthalSukochev2003}, they observe that the above question is equivalent to the following one:
\begin{quote}
    Let \((f_n)\) be a seminormalized basic sequence in \(L^p(\mathcal N)\),
\(p\) and \(\mathcal N\) as above.
Does \((f_n)\) have a subsequence \((f'_n)\) which is dominated by
the \(\ell^p\)-basis? i.e., such that
\(\sum_j c_jf'_j\) converges in \(L^p(\mathcal N)\) whenever
\(\sum_j |c_j|^p<\infty\)? \cite[Section~5, Problem$'$]{HaagerupRosenthalSukochev2003}
\end{quote}

We answer the above questions affirmatively. More precisely, we show that 
Haagerup $L_p(\M)$ has  uniform USP for every $1<p<\infty$
(see Corollary~\ref{cor:nclp-usp} below).
As application of Corollary~\ref{cor:nclp-usp}, we give an upper
$\ell_p$ estimate on a subsequence of seminormalized weakly null sequence in $L_p(\M)$, $1<p<2$ (see Corollary~\ref{cor:hrs} below), which answer the above Haagerup--Rosenthal-Sukochev problems in the affirmative. 
We also establish a weak-type $(1,1)$
 estimate  for seminormalized weakly null
sequences in \rev{semifinite} noncommutative $L_1$-spaces
(see Corollary~\ref{cor:weak-endpoint} below).
Moreover, we establish 
the Rosenthal-type inequalities of
Junge, Parcet and Xu \cite[Theorem~A]{JungeParcetXu2007} (see also \cite[Theorem~1.1]{JiaoSukochevZaninZhou2017}) for weakly null sequences. 



Recall a familiar result due to Rosenthal \cite{Rosenthal1975} (and a more general result due to Bourgain \cite{Bourgain1980}) that  every non-Dunford--Pettis(DP)-operator on $L_1$ fixes a copy of $\ell_2$.
This result was extended to the setting of hyperfinite finite von Neumann algebras in \cite{HuangPlievSukochev2022}. The general case remains an open  problem:
\begin{quote}
    whether every $l_2$-strictly singular operator on $M_*$ is automatically (DP)?\cite{HuangPlievSukochev2022}
\end{quote}
One of the applications of Theorem \ref{thm:main} is to answer this question in the affirmative (see Proposition~\ref{prop:dp-l2ss-arbitrary} below).

\section{Preliminaries}\label{sec:prelim}

We recall the notation and background results which are  needed
in this paper. General background
on von Neumann algebras and Haagerup \(L_1\)-spaces can be found in
\cite{HaagerupJungeXu2010,Haagerup1979,Takesaki1979}.
We write $\one$ for the identity element of a von Neumann algebra \(\M\).

\subsection{Measurable operators}

In the semifinite setting, \(\M\) denotes a von Neumann algebra equipped with a normal faithful semifinite trace \(\tau\). We use the standard space of \(\tau\)-measurable
operators affiliated with \(\M\), as in
\cite{FackKosaki1986,Nelson1974}.  For \(0<p<\infty\),
\[
  L_p(\M,\tau)
  =\Bigl\{x:\tau\bigl(|x|^p\bigr)<\infty\Bigr\},\qquad
  \|x\|_p=\tau\bigl(|x|^p\bigr)^{1/p}.
\]
If \(\mu_t(x)\) denotes the generalized singular-value function, then,
by \cite[Corollary~2.8]{FackKosaki1986},
\[
  \|x\|_p^p=\int_0^\infty \mu_t(x)^p\,dt.
  \label{eq:singular-integral}
\]
Following \cite[p.~231]{ParcetRandrianantoanina2006}, the weak
\(L_1\) quasi-norm is
\[
  \|x\|_{1,\infty}=\sup_{t>0}t\mu_t(x).
  \label{eq:weak-L1}
\]
We shall use some standard facts about the quasi-Banach range \(0<p<1\).
By
\cite[Theorem~4.9(iii)]{FackKosaki1986}, we have that
\begin{align}\label{eq:r-triangle}
  \|x+y\|_p^p\leq \|x\|_p^p+\|y\|_p^p.
\end{align}

\rev{More generally, for $0<r\le1$, an $r$-normed
quasi-Banach space is a complete quasi-normed
space $Y$ satisfying}
\begin{align}
 \|y+z\|_Y^r
 \le \|y\|_Y^r+\|z\|_Y^r
 \qquad (y,z\in Y).
 \label{eq:r-triangle}
\end{align}

In particular, by
\cite[Theorem~4.9(iii)]{FackKosaki1986},
$L_r(\M,\tau)$ is an $r$-normed
quasi-Banach space for $0<r<1$.
Moreover, when $\tau(1)=1$, we have
\begin{align}\label{eq:L1-Lp}
 \|x\|_r\le\|x\|_s,
 \qquad 0<r<s\le\infty.
\end{align}
and it follows from \cite[Corollary~2.8]{FackKosaki1986} that
\begin{align}
  \|x\|_p^p
  \leq \frac{1}{1-p}\|x\|_{1,\infty}^p,
  \qquad 0<p<1.
\label{eq:weak-embed}
\end{align}


The following inequality combines the noncommutative
\(L_p\)-cotype inequality \cite[Corollary~5.8]{PisierXu2003} with the
quasi-Banach Kahane--Khintchine inequality
\cite[Proposition~2.1]{Kalton2005}. 

\begin{lemma}\label{lem:cotype}
Let \(\M\) be a von Neumann algebra equipped with a normal faithful
semifinite trace \(\tau\).
For
every \(0<p<1\) there is
\(K_p<\infty\) such that every finite family
\((z_k)\subset L_p(\M,\tau)\) satisfies
\begin{align}
  \left(\sum_k\|z_k\|_p^2\right)^{1/2}
  \leq K_p
  \left(
    \E_\varepsilon
    \left\|\sum_k\varepsilon_kz_k\right\|_p^p
  \right)^{1/p},
  \label{eq:cotype}
\end{align}
where the \(\varepsilon_k\)'s are independent symmetric signs, that is,
independent random variables satisfying
\(\mathbb P(\varepsilon_k=1)=\mathbb P(\varepsilon_k=-1)=\tfrac12\).
Here \(\E_\varepsilon\) denotes expectation with respect to these signs.  
\end{lemma}

\subsection{Martingale difference sequence}
We recall some standard terminology concerning
martingales and predictable sequences
\cite[Section~2]{GirardiWeis2005}.

Let $X$ be a Banach space. A \emph{filtration}
on a probability space $(\Omega,\mathcal F,\mathbb P)$
is a sequence $(\mathcal F_k)_{k\ge0}$ of
$\sigma$-algebras satisfying
\[
 \mathcal F_k\subseteq\mathcal F_{k+1}
 \subseteq\mathcal F
 \qquad(k\ge0).
\]
A sequence $(d_k)_{k\ge1}\subset L_1(\Omega;X)$
is called an \emph{$X$-valued martingale
difference sequence} if
\[
 d_k\text{ is }\mathcal F_k\text{-measurable},
 \qquad
 \mathbb E\bigl(d_k\mid\mathcal F_{k-1}\bigr)=0
 \qquad(k\ge1).
\]

A \emph{sign sequence} is a sequence
$(\gamma_k)$ with $\gamma_k\in\{-1,1\}$.
A sequence $(\xi_{k-1})_{k\ge1}$ of random
variables is called a \emph{predictable
sign sequence} if
\[
 \xi_{k-1}\in\{-1,1\}\quad\text{almost surely},
 \qquad
 \xi_{k-1}\text{ is }\mathcal F_{k-1}
 \text{-measurable}
 \qquad(k\ge1).
\]

A sequence $(\varepsilon_k)_{k\ge1}$ is called a sequence
of \emph{independent symmetric signs} if its members
are independent random variables satisfying
\[
 \mathbb P(\varepsilon_k=1)
 =\mathbb P(\varepsilon_k=-1)
 =\frac12
 \qquad(k\ge1)
\]
\cite[Chapter~1, Section~5, p.~5]{Kahane1985}.

A filtration of von Neumann subalgebras of \(\M\) is a sequence
\((\M_k)_{k\geq0}\) satisfying
\[
 \M_k\subseteq\M_{k+1}\subseteq\M\qquad(k\geq0).
\]
Let \((\M_k)_{k\geq0}\) be such a filtration of semifinite von Neumann
subalgebras for which the trace-preserving conditional
expectations \(\mathcal E_k:\M\to\M_k\) exist.  A sequence
\((d_k)_{k\geq1}\subset L_1(\M,\tau)\) is a martingale difference sequence
if \(d_k\in L_1(\M_k)\) and \(\mathcal E_{k-1}d_k=0\). 

We record the following finite version of
\cite[Theorem~3.1]{ParcetRandrianantoanina2006}.

\begin{theorem}
\label{thm:MT}
Let \(\M\) be a von Neumann algebra equipped with a normal faithful
semifinite trace \(\tau\), and let \((\M_k)_{k\geq0}\) be an increasing
filtration of semifinite von Neumann subalgebras for which the
trace-preserving conditional expectations
\(\mathcal E_k:\M\to\M_k\) exist.
There is an absolute constant \(C_{\mathrm{MT}}\) with the following
property:  if \((d_k)_{k=1}^m\) is a finite martingale difference sequence
and
\[
 \xi_0=1,\qquad
 \xi_{k-1}\in\M_{k-1}\cap\M_k',\qquad
 \|\xi_{k-1}\|_\infty\leq1
 \quad(1\leq k\leq m),
\]
then
\begin{align}
 \left\|\sum_{k=1}^m\xi_{k-1}d_k\right\|_{1,\infty}
 \leq C_{\mathrm{MT}}
 \left\|\sum_{k=1}^m d_k\right\|_1.
 \label{eq:MT}
\end{align}
\end{theorem}
\begin{proof}
  Put
  \[
  x_j=\sum_{k=1}^j d_k\qquad(1\leq j\leq m).
\]
Then \((x_j)_{j=1}^m\) is a martingale, \(dx_j=d_j\), and
\(x_j=\mathcal E_jx_m\).  Contractivity of \(\mathcal E_j\) on \(L_1\),
together with the case \(j=m\), gives
\[
  \sup_{1\leq j\leq m}\|x_j\|_1=\sup_{1\leq j\leq m}\|\mathcal E_jx_m\|_1
  =\|x_m\|_1
  =\left\|\sum_{k=1}^m d_k\right\|_1.
\]
Thus \eqref{eq:MT} is precisely the consequence of
\cite[Theorem~3.1]{ParcetRandrianantoanina2006}.
\end{proof}




\subsection{Haagerup
\texorpdfstring{\(L_p\)}{Lp}-spaces}
\label{Haagerup}

Let \(\M\) be an arbitrary von Neumann algebra and
fix a normal faithful semifinite weight \(\psi\).
 Let
\[
 \Core:=\M\rtimes_{\sigma^\psi}\mathbb R
\]
be the crossed product, let $(\theta_s)_{s\in\mathbb R}$ be its dual
action, and let $\Trc$ be the canonical normal faithful semifinite
trace on $\Core$, satisfying
\begin{align*}
  \Trc\circ\theta_s=e^{-s}\Trc
  \qquad(s\in\mathbb R).
\end{align*}
We write \(L_0(\Core,\Trc)\) for the space of all
\(\Trc\)-measurable operators affiliated with \(\Core\);
see \cite[Section~1.2]{HaagerupJungeXu2010}.

Since \(\Core\) is semifinite, we also consider the
usual tracial \(L_p\)-spaces
\cite{Nelson1974,FackKosaki1986}:
\[
  L_p(\Core,\Trc)
  =
  \left\{
    x\in L_0(\Core,\Trc):
    \Trc\bigl(|x|^p\bigr)<\infty
  \right\},
  \qquad 0<p<\infty,
\]
equipped with the norm (or quasi-norm when \(0<p<1\))
\[
  \|x\|_{L_p(\Core,\Trc)}
  =\Bigl(\Trc\bigl(|x|^p\bigr)\Bigr)^{1/p}.
\]
In particular, \(L_1(\Core,\Trc)\) denotes the
tracial \(L_1\)-space associated with the
semifinite algebra \((\Core,\Trc)\).

For \(0<p<\infty\), the Haagerup space \(L_p(\M)\)
is defined by
\[
  L_p(\M)=
  \left\{
    x\in L_0(\Core,\Trc):
    \theta_t(x)=e^{-t/p}x
    \text{ for every }t\in\mathbb R
  \right\}.
  \label{eq:haagerup-Lp}
\]
Let \(\operatorname{tr}_{\M}\) denote the canonical
trace functional on \(L_1(\M)\). The \(L_p\)-norm
(or quasi-norm when \(0<p<1\)) is given by
\[
  \|x\|_p
  =\Bigl(\operatorname{tr}_{\M}\bigl(|x|^p\bigr)\Bigr)^{1/p},
  \qquad x\in L_p(\M).
  \label{eq:haagerup-Lp-norm}
\]
These spaces are independent of the choice of
\(\psi\), up to canonical isometric identification;
see \cite{Haagerup1979,PisierXu2003} and
\cite[Section~1.2]{HaagerupJungeXu2010}.

In particular, \(L_1(\M)\) is canonically and
isometrically identified with \(\M_*\), and
\begin{align}
  \|x\|_1=\operatorname{tr}_{\M}(|x|),
  \qquad
  \langle x,a\rangle=\operatorname{tr}_{\M}(xa)
  \quad(x\in L_1(\M),\ a\in\M).
  \label{eq:haagerup-pairing}
\end{align}
Moreover,
\begin{align}\label{tr le 1 infty}
  \bigl|\operatorname{tr}_{\M}(xa)\bigr|
  \leq\|x\|_1\|a\|_\infty,
  \qquad
  \|xa\|_1\leq\|x\|_1\|a\|_\infty.
\end{align}
The polar partial isometry of an element of
\(L_p(\M)\) belongs to \(\M\), and H\"older's
inequality holds for products of Haagerup
\(L_p\)-spaces
\cite[Section~1.2]{HaagerupJungeXu2010}. 
Following \cite[Definition~1.1 and Remark~2.8]{RaynaudXu2003},
two elements $x,y\in L_p(\M)$ are called \emph{left disjoint}
or \emph{right disjoint} if, respectively,
\[
 x^*y=0
 \quad\text{or}\quad
 xy^*=0.
\]
They are called \emph{disjoint} if both conditions hold.
When \(\M\) is semifinite, these spaces agree
isometrically with the usual tracial \(L_p\)-spaces
\cite{Haagerup1979,PisierXu2003}.


If \(\M\) is a \(\sigma\)-finite von Neumann algebra
with a normal faithful state \(\varphi\). 
Let \(D=h_\varphi\in L_1(\M)_+\) be the Haagerup
density of \(\varphi\). Then \(D\) has full support,
and
\begin{align}\label{theta_t}
  \varphi(a)=\operatorname{tr}_{\M}(Da)
  \quad(a\in\M),
  \qquad
  \theta_t(D)=e^{-t}D
  \quad(t\in\mathbb R).
\end{align}

Let \(\Phi:\Core_+\to\widehat{\M}_+\) denote the
canonical operator-valued weight, given by
\[
  \Phi(z)=\int_{\mathbb R}\theta_t(z)\,dt
  \qquad(z\in\Core_+);
\]
see \cite[(1.4)]{HaagerupJungeXu2010}.
The dual weight 
\begin{align*}
    \widehat\varphi=\varphi\circ\Phi
\end{align*}
satisfies
\begin{align*}
  \widehat\varphi(z)
  =\Trc\bigl(D^{1/2}zD^{1/2}\bigr)
  \qquad(z\in\Core_+),
\end{align*}
by \cite[Sections~1.1--1.2]{HaagerupJungeXu2010}.
\rev{More generally, for every \(\omega\in\M_*^+\), 
letting \(h_\omega\) be the Haagerup density of $\omega$, by \cite[Definition~1.1]{Haagerup1979}, 
\begin{align}\label{eq:dual-weight-density}
 \omega\bigl(\Phi(b)\bigr)
 =\widehat\omega(b)
=\Trc\bigl(h_\omega^{1/2}bh_\omega^{1/2}\bigr)
 \qquad(b\in\Core_+).
\end{align}}

For every bounded nonnegative Borel function \(f\)
such that
\[
  \int_0^\infty f(s)\,\frac{ds}{s}<\infty,
\]
\eqref{theta_t} and spectral calculus give
\begin{align}
  \Phi\bigl(f(D)\bigr)
  =
  \left(
    \int_0^\infty f(s)\,\frac{ds}{s}
  \right)\one.
  \label{eq:Phi-functional}
\end{align}

For \(N\in\mathbb N\), define the spectral projection
\begin{align}
  P_N=\one_{[e^{-N},e^N)}(D).
  \label{eq:PN-def}
\end{align}
Taking \(x=h_\omega\in L_1(\M)_+\) and \(b=P_NzP_N\) in
\eqref{eq:dual-weight-density}, we obtain
\[
 \Trc(P_NxP_Nz)
 =\bigl\langle x,\Phi(P_NzP_N)\bigr\rangle
 \qquad(x\in L_1(\M)_+,\ z\in\Core_+).
\]
Then, it follows from the linearity of $\Trc$ that
\begin{align}
  \Trc(P_NxP_Nz)
  =\bigl\langle x,\Phi(P_NzP_N)\bigr\rangle
  \quad
  (x\in L_1(\M),\ z\in\Core).
  \label{eq:window-weight-duality}
\end{align}
Together with \eqref{eq:Phi-functional} and \eqref{eq:haagerup-pairing},
we obtain that
\begin{align}
\Trc(P_NxP_N)=\bigl\langle x,\Phi(P_N)\bigr\rangle
  \overset{\eqref{eq:Phi-functional}}{=}\bigl\langle x,2N\one\bigr\rangle
\overset{\eqref{eq:haagerup-pairing}}{=}2N\|x\|_1
  \qquad(x\in L_1(\M)_+).
  \label{eq:positive-window-mass}
\end{align}

\begin{proposition}
\label{prop:core-big-window}
Let \(\M\) be a \(\sigma\)-finite von Neumann algebra,
and retain the notation \(\varphi,\Core,\Trc,D,P_N\)
introduced above.
For \(j\in\mathbb Z\), set
\begin{align}
 p_j=\one_{[e^j,e^{j+1})}(D).
 \label{eq:log-blocks}
\end{align}
Let \(y=D^{1/2}aD^{1/2}\), where \(a\in\M\) and
\(\|a\|_\infty\leq C\). 
Then
\begin{align}
 \|p_iyp_j\|_{L_1(\Core,\Trc)}
 \leq(e-1)C e^{-|i-j|/2}
 \qquad(i,j\in\mathbb Z),
 \label{eq:block-decay}
\end{align}
and
\begin{align}
 \frac1{2N}\|P_NyP_N\|_{L_1(\Core,\Trc)}
 \geq\|y\|_{L_1(\M)}-\frac{B_0C}{2N},
 \label{eq:big-window-lower}
\end{align}
where \begin{align}
 B_0=\frac{2(e-1)e^{-1/2}}{\bigl(1-e^{-1/2}\bigr)^2}.
 \label{eq:B0}
\end{align}
\end{proposition}

\begin{proof}
Since \(D\in L_1(\mathcal M)_+\) and \(\|D\|_1=1\), it follows from \cite[Lemma 4.8]{FackKosaki1986} that
\[
  \Trc\bigl(\one_{(s,\infty)}(D)\bigr)=s^{-1}
  \qquad(s>0).
\]
Consequently,
\begin{align}
 \Trc(p_j)=\bigl(1-e^{-1}\bigr)e^{-j}.
 \label{eq:pj-trace}
\end{align}

Since \(y=D^{1/2}aD^{1/2}\), where \(a\in\M\) and
\(\|a\|_\infty\leq C\), we have that
\begin{align}
 \|p_iyp_j\|_\infty
 &=\bigl\|p_iD^{1/2}aD^{1/2}p_j\bigr\|_\infty
 \label{eq:block-infty}\\
 &\le \bigl\|p_iD^{1/2}\bigr\|_\infty\|a\|_\infty\bigl\|D^{1/2}p_j\bigr\|_\infty
 \stackrel{\eqref{eq:log-blocks}}{\leq}
 Ce^{(i+1)/2}e^{(j+1)/2}. \notag
\end{align}
Then, it follows from H\"older's inequality that
\[
 \|p_iyp_j\|_{L_1(\Core,\Trc)}\leq\|p_iyp_j\|_\infty
 \min\bigl\{\Trc(p_i),\Trc(p_j)\bigr\}\stackrel{\eqref{eq:pj-trace},\eqref{eq:block-infty}}{\le}(e-1)C e^{-|i-j|/2},
\]
which proves
\eqref{eq:block-decay}.

Summing \eqref{eq:block-decay} over the blocks outside
the spectral window gives
\begin{align}\label{sum q_i}
    \sum_{i=-N}^{N-1}\sum_{j=N}^{\infty}\|p_iyp_j\|_{L_1(\Core,\Trc)}
    \leq (e-1)C\frac{e^{-1/2}}{\bigl(1-e^{-1/2}\bigr)^2},
\end{align}
and the same estimate holds for \(j\leq -N-1\). Hence,
\begin{eqnarray}\label{eq:cross-window}
   \bigl\|P_Ny(1-P_N)\bigr\|_{L_1(\Core,\Trc)}&\stackrel{\eqref{eq:PN-def}}{=}&\left\|\left(\sum_{i=-N}^{N-1} p_i\right)y\left(1-\sum_{j=-N}^{N-1} p_j\right)\right\|_{L_1(\Core,\Trc)}\\ 
&\leq&\sum_{i=-N}^{N-1}
 \left(
   \sum_{j=N}^{\infty}\|p_iyp_j\|_{L_1(\Core,\Trc)}
   +\sum_{j=-\infty}^{-N-1}\|p_iyp_j\|_{L_1(\Core,\Trc)}
 \right)\nonumber\\ 
  &\stackrel{\eqref{sum q_i}}{\le}&\frac{2(e-1)e^{-1/2}}{\bigl(1-e^{-1/2}\bigr)^2}C\stackrel{\eqref{eq:B0}}{=} B_0C.\nonumber
\end{eqnarray}
Now write the Haagerup polar decomposition as \(y=u|y|\). Since \(u\) is a partial isometry and \(P_N\) is a
projection, it follows that 
\begin{align}\label{u^*P_N}
\|u^*P_N\|_\infty\leq1.
\end{align}
Then, we have that
\begin{eqnarray*}
\|P_NyP_N\|_{L_1(\Core,\Trc)}
    &\stackrel{
    \eqref{u^*P_N}}{\ge}&
    \bigl|\Trc(P_NyP_Nu^*P_N)\bigr| 
    \\&\geq&
    \Trc(P_N|y^*|P_N)
    -\bigl\|P_Ny(1-P_N)u^*P_N\bigr\|_{L_1(\Core,\Trc)} \\
    &\stackrel{\eqref{u^*P_N}}{\ge}& \Trc(P_N|y^*|P_N)
    -\bigl\|P_Ny(1-P_N)\bigr\|_{L_1(\Core,\Trc)}\\
    &\stackrel{\eqref{eq:positive-window-mass},\eqref{eq:cross-window}}{\ge}&
    2N\|y\|_{L_1(\M)}-B_0C.
\end{eqnarray*}
This proves \eqref{eq:big-window-lower}.
\end{proof}

\subsection{Weak compactness}

For a Banach space \(X\), a subset \(K\subset X\) is called
\emph{relatively weakly compact} if its closure \(\overline K^{\,w}\)
in the weak topology \(\sigma(X,X^*)\) is weakly compact.

A sequence $(x_n)$ in a Banach space $X$ is said to
\emph{converge weakly} to $x\in X$, written
$x_n\rightharpoonup x$, if
\[
 f(x_n)\longrightarrow f(x)
 \qquad\text{for every }f\in X^*.
\]
The sequence is called \emph{weakly null} if
$x_n\rightharpoonup0$.

In the  setting of finite-trace von Neumann algebras, the following characterization
of relative weak compactness is given in
\cite[Corollary~2.6]{HaagerupRosenthalSukochev2003}.

\begin{proposition}\cite[Corollary~2.6]
{HaagerupRosenthalSukochev2003}\label{prop:equi}
Let \(\M\) be a von Neumann algebra equipped with a normal faithful
finite trace \(\tau\).
Then \(K\subset L_1(\M,\tau)\) is relatively weakly compact if and only if
\begin{align}
 \lim_{R\to\infty}\sup_{x\in K}
 \tau\!\left(|x|\one_{(R,\infty)}(|x|)\right)=0.
 \label{eq:equi-spectral}
\end{align}
\end{proposition}

We also recall the following consequence of the
weak compactness criterion
\cite[Theorem~4.14]{RaynaudXu2003}.

\begin{lemma}\label{lem:weak-compact-compression}
Let \(\M\) be an arbitrary von Neumann algebra,
and let \(K\subset\M_*\) be relatively weakly compact.
Then, for every net of projections
\(e_\alpha\to0\) in the weak-star topology,
\begin{align}
  \sup_{x\in K}\|e_\alpha x e_\alpha\|_1
  \longrightarrow0.
  \label{eq:typefree-diagonal-equi}
\end{align}
Here the norm is understood through the canonical
identification \(\M_*=L_1(\M)\).
\end{lemma}

\subsection{Equiintegrability}
We next recall the equiintegrability notions of Raynaud and Xu
\cite[Definition~4.10]{RaynaudXu2003}.
A bounded set $K\subset L_p(\M)$ is called
\emph{$p$-equiintegrable} if
\[
 \sup_{x\in K}\|e_\alpha x e_\alpha\|_p\longrightarrow0
\]
for every net of projections $e_\alpha\to0$ in the weak-star topology.
A bounded set
\(K\subset L_p(\M)\) is \(p\)-\emph{left equiintegrable} (respectively, \(p\)-right equiintegrable, \(p\)-\emph{biequiintegrable}) if
\[
 \sup_{x\in K}\|e_\alpha x\|_p\longrightarrow0
\]
(respectively, \(
 \sup_{x\in K}\| xe_\alpha\|_p\longrightarrow0
\), \(
\max\bigl\{ \sup_{x\in K} \|e_\alpha x  \|_p,\sup_{x\in K}   \| x  e_\alpha \|_p\bigr\} \longrightarrow0
\)) for every net of projections \(e_\alpha\to0\) in the weak-star topology.  

We shall use the following characterization of 
\(p\)-equiintegrability.

\begin{lemma}\cite[Proposition~4.13(i)--(ii)]{RaynaudXu2003}
\label{lem:equi-approx}
Let \(0<p<\infty\), and let
\(K\subset L_p(\M)\) be bounded.
Then the following assertions hold.
\begin{enumerate}
\item
The set \(K\) is left (respectively, right)
\(p\)-equiintegrable if and only if, for every
\(\varepsilon>0\), there exists
\(h_\varepsilon\in L_p(\M)\) such that
\[
  \sup_{x\in K}
  \inf_{\substack{a\in\M\\
  \|a\|_\infty\leq1}}
  \|x-h_\varepsilon a\|_p
  <\varepsilon
\]
(respectively,
\[
  \sup_{x\in K}
  \inf_{\substack{a\in\M\\
  \|a\|_\infty\leq1}}
  \|x-ah_\varepsilon\|_p
  <\varepsilon).
\]

\item
The set \(K\) is \(p\)-equiintegrable
if and only if, for every
\(\varepsilon>0\), there exists
\(h_\varepsilon\in L_p(\M)\) such that
\[
  \sup_{x\in K}
  \inf_{\substack{a,b\in\M\\
  \|a\|_\infty,\|b\|_\infty\leq1}}
  \|x-ah_\varepsilon-h_\varepsilon b\|_p
  <\varepsilon.
\]
\end{enumerate}
If \(\M\) is \(\sigma\)-finite, one may take
\(h_\varepsilon=C_\varepsilon h_0\), where
\(C_\varepsilon>0\) and
\(h_0\in L_p(\M)_+\) is any fixed element
with full support.
\end{lemma}

\begin{lemma}
\label{lem:uniform-symmetric-factorization}
Let \(\M\) be a \(\sigma\)-finite von Neumann algebra, let
\(\varphi\) be a normal faithful
state with Haagerup density \(D\), and let $K\subset L_1(\M)$ be bounded and $1$-biequiintegrable.  For every
\(\varepsilon>0\) there is \(C_\varepsilon<\infty\) such that, for every
\(x\in K\), one can find \(m_x\in\M\) satisfying
\begin{align*}
 \|m_x\|_\infty\leq C_\varepsilon,
\qquad
 \bigl\|x-D^{1/2}m_xD^{1/2}\bigr\|_1<\varepsilon.
\end{align*}
\end{lemma}

\begin{proof}
Without loss of generality,
we may assume that \(\sup_{x\in K}\|x\|_1\le 1\).
For each \(x\in K\), write \(x=u_x|x|\), the polar decomposition,  and set
\[
  a_x=u_x|x|^{1/2},\qquad b_x=|x|^{1/2}.
\]
Then 
\begin{align}\label{eq:x=a_xb_x}
x=a_xb_x,
\end{align}
and
\begin{align}
  \|a_x\|_2^2=\|b_x\|_2^2
  =\|x\|_1\le 1.
  \label{eq:symmetric-L2-bound}
\end{align}
For every projection \(e\in\M\), it follows from \eqref{tr le 1 infty} that
\begin{align}\label{eq:symmetric-L2-equi}
  \|ea_x\|_2^2&=\operatorname{tr}_{\M}\bigl((ea_x)(ea_x)^*\bigr)
  =\operatorname{tr}_{\M}(exu_x^*)
  \overset{\eqref{tr le 1 infty}}{\le}\|ex\|_1, \\
  \|b_xe\|_2^2&=\operatorname{tr}_{\M}\bigl((b_xe)^*(b_xe)\bigr)
  =\operatorname{tr}_{\M}(u_x^*xe)
  \overset{\eqref{tr le 1 infty}}{\le}\|xe\|_1.
  \notag
\end{align}
By \eqref{eq:symmetric-L2-equi} and the
$1$-biequiintegrability of \(K\), the families
\(\{a_x:x\in K\}\) and \(\{b_x:x\in K\}\)
are respectively left and right
\(2\)-equiintegrable.

Fix \(\varepsilon>0\), and choose \(\eta>0\)
such that
\begin{align}\label{choose eta < varepsilon}
  \eta(2+\eta)<\varepsilon.
\end{align}
Since \(\varphi\) is faithful, \(D^{1/2}\)
has full support. 
Applying \lemref{lem:equi-approx}(1) with \(p=2\) and
\(h_0=D^{1/2}\) to \(\{a_x:x\in K\}\) and \(\{b_x:x\in K\}\),
we obtain constants \(C_\eta^{(1)},C_\eta^{(2)}>0\) such that,
for every \(x\in K\), there exist
\(v_x,w_x\in\M\) satisfying
\begin{align}
  \|v_x\|_\infty\le C_\eta^{(1)},\|w_x\|_\infty&\le C_\eta^{(2)}
 \label{eq:symmetric-L2-approx}\\
  \bigl\|a_x-D^{1/2}v_x\bigr\|_2
  &<\eta, \notag\\
  \bigl\|b_x-w_xD^{1/2}\bigr\|_2
  &<\eta.\notag
\end{align}
Then it follows from triangle inequality that
\begin{align}
  \bigl\|D^{1/2}v_x\bigr\|_2
  \le\|a_x\|_2
    +\bigl\|a_x-D^{1/2}v_x\bigr\|_2
  \stackrel{\eqref{eq:symmetric-L2-bound},
  \eqref{eq:symmetric-L2-approx}}{<}
    1+\eta.
  \label{eq:symmetric-L2-middle}
\end{align}
By H\"older's inequality
\cite[Section~1.2]{HaagerupJungeXu2010},
we have that
\begin{eqnarray*}
  \bigl\|x-D^{1/2}(v_xw_x)D^{1/2}\bigr\|_1&\stackrel{\eqref{eq:x=a_xb_x}}{=}&\bigl\|a_xb_x-D^{1/2}(v_xw_x)D^{1/2}\bigr\|_1\\
  &\le&
    \Bigl\|\bigl(a_x-D^{1/2}v_x\bigr)b_x\Bigr\|_1
    +\Bigl\|D^{1/2}v_x\bigl(b_x-w_xD^{1/2}\bigr)\Bigr\|_1
    \\
  &\le&
    \bigl\|a_x-D^{1/2}v_x\bigr\|_2\|b_x\|_2
    +\bigl\|D^{1/2}v_x\bigr\|_2
       \bigl\|b_x-w_xD^{1/2}\bigr\|_2\notag
       \\&\stackrel{\eqref{eq:symmetric-L2-approx}}{<}&
       \eta\Bigl(\|b_x\|_2+\bigl\|D^{1/2}v_x\bigr\|_2\Bigr)\notag
\stackrel{
    \eqref{eq:symmetric-L2-bound},
    \eqref{eq:symmetric-L2-middle}}{<}
    \eta(2+\eta)
    \stackrel{\eqref{choose eta < varepsilon}}{<}\varepsilon.
  \notag
\end{eqnarray*}
\rev{Moreover, setting \(m_x=v_xw_x\) and
\(
  C_\varepsilon= C_\eta^{(1)} C_\eta^{(2)}
\),}
we have that
\[
  \|m_x\|_\infty
  \le\|v_x\|_\infty\|w_x\|_\infty
  \stackrel{\eqref{eq:symmetric-L2-approx}}{\le} C_\varepsilon.
\]
This completes the proof.
\end{proof}

\subsection{\(2\)-co-lacunary sequences}

We call a sequence \((x_n)\) \emph{seminormalized} when
\[0<\inf_n\|x_n\|\leq\sup_n\|x_n\|<\infty.\]
A sequence \((x_n)\) in a Banach space \(X\) is called
\emph{weakly null} if \(x_n\weakto0\), that is, if
\(f(x_n)\to0\) for every \(f\in X^*\).  
For \(N\geq1\), the sequence \((x_n)_{n\geq N}\) is called a \emph{tail}
of \((x_n)\).

A sequence \((x_n)\) in a Banach space \(X\) is called a
\emph{basic sequence} if it is a Schauder basis for its closed linear span
\([x_n]=\overline{\operatorname{span}}\{x_n:n\geq1\}\).
Its \emph{basis constant} is \(K_b=\sup_{m\geq1}\|P_m\|\), where
\[
 P_m\left(\sum_{n=1}^{\infty}a_nx_n\right)
 =\sum_{n=1}^m a_nx_n.
\]
\rev{Every seminormalized weakly null sequence admits a basic
subsequence with basis constant less than $2$
\cite{BessagaPelczynski1958}.}
After passing to this subsequence and relabeling it as \((x_n)\), put
\(\delta=\inf_n\|x_n\|_X>0\). Then, with \(B=2K_b/\delta\),
\begin{equation}\label{eq:coefficient-bound}
 \max_{1\leq j\leq m}|a_j|
 \leq B\left\|\sum_{j=1}^m a_jx_{n_j}\right\|_X
 \qquad(n_1<\cdots<n_m)
\end{equation}
for every \(m\geq1\) and all scalars \(a_1,\ldots,a_m\).

We use the terminology of
\cite[Definition~1]{AldousFremlin1982}; see also
\cite[Section~3.3]{ParcetRandrianantoanina2006}.

\begin{definition}\label{def:colac}
A sequence \((x_n)\) in a normed space \(X\) is \(2\)-co-lacunary if there
is \(\gamma>0\) such that
\begin{align*}
  \left\|\sum_n a_nx_n\right\|_X
  \geq\gamma\left(\sum_n|a_n|^2\right)^{1/2}
\end{align*}
for every finitely supported scalar sequence \((a_n)\).
\end{definition}

We shall need two properties of \(2\)-co-lacunary sequences,
due to Aldous and Fremlin \cite[Lemma~2]{AldousFremlin1982}.

\begin{lemma}\cite[Lemma~2]{AldousFremlin1982}\label{lem:co-lacunary}
\begin{enumerate}
\item Suppose \((u_n)\) is \(2\)-co-lacunary and
\(\sum_n\|u_n-v_n\|<\infty\).  Then a tail of \((v_n)\) is
\(2\)-co-lacunary.
\item If \((u_n)\) is a bounded \(2\)-co-lacunary sequence in \(X\) and \(u\in X\),
then some tail of \((u_n+u)\) is \(2\)-co-lacunary.
\end{enumerate}
\end{lemma}



We shall also use the following four-corner
subsequence decomposition of Raynaud and Xu
\cite[Corollary~4.12]{RaynaudXu2003}.

\begin{lemma}\label{lem:four-corners}
Let \(\M\) be an arbitrary von Neumann algebra,
and let \((x_n)\) be a bounded sequence in
\(L_1(\M)\). After passing to a subsequence
and relabeling it as \((x_k)\), there exist
pairwise disjoint projections \((p_k)\subset\M\)
such that, with
\begin{align}
  \begin{aligned}
    a_k&=p_k^\perp x_kp_k^\perp,&
    b_k&=p_k^\perp x_kp_k,\\
    c_k&=p_kx_kp_k^\perp,&
    d_k&=p_kx_kp_k,
  \end{aligned}
  \label{eq:four-corners}
\end{align}
the sequence \((a_k)\) is $1$-biequiintegrable,
\((b_k)\) is $1$-left equiintegrable and right
disjoint, \((c_k)\) is $1$-right equiintegrable
and left disjoint, and \((d_k)\) is disjoint.
\end{lemma}

\subsection{Well-founded trees and prunings}
\label{sec:tree-prelim}

We recall the tree terminology which will be  used in the proof of
the finite-trace case, following
\cite{AlspachJuddOdell2005,CauseyDoebele2020}.

A \emph{tree} is a partially ordered set
\((W,\preceq)\) such that, for each \(t\in W\),
the predecessor set
\[
  \{s\in W:s\preceq t\}
\]
is finite and linearly ordered
\cite[Definition~2.1]{AlspachJuddOdell2005}.
The elements of \(W\) are called \emph{nodes}.
If \(s\preceq t\), then \(s\) is a
\emph{predecessor}, or \emph{ancestor}, of \(t\).
The tree is \emph{rooted} if it has a least node
\(r_W\), called its \emph{root}.
If \(s\prec t\) with no node strictly between them,
then \(t\) is an \emph{immediate child} of \(s\). 
A node \(t\succeq s\) is a \emph{descendant} of \(s\),
and
\[
  W_s=\{t\in W:s\preceq t\}
\]
is the descendant tree at \(s\);
cf.\ \cite[Section~2]{CauseyDoebele2020}.
The \emph{depth} of a node \(t\in W\) is defined by
\[
  \operatorname{depth}(t)
  =\#\{s\in W:s\prec t\}.
\]
Thus the root has depth \(0\), and every immediate
child of a node of depth \(k\) has depth \(k+1\).
A maximal node is called \emph{terminal}, or a
\emph{leaf}, and \(\Term(W)\) denotes the set of
terminal nodes. All other nodes are \emph{internal}. 

A \emph{branch} is a linearly ordered subset of \(W\).
The tree \(W\) is \emph{well founded} if it has no
infinite strictly increasing branch
\cite[Proposition~2.1]{CauseyDoebele2020}. \rev{For a well-founded tree, the node rank of a node $s$ is defined recursively by}
\[
  \rho_W(s)
  =\sup\{\rho_W(t)+1:
  t\text{ is an immediate child of }s\},   \qquad \sup\varnothing=0.
\]


\begin{center}
\resizebox{0.98\linewidth}{!}{%
\begin{tikzpicture}[
  x=1cm, y=1cm,
  heading/.style={anchor=west, font=\small\sffamily\bfseries},
  note/.style={anchor=west, font=\footnotesize},
  vertex/.style={circle, draw=black!70, fill=white,
                 minimum size=6mm, inner sep=1pt, font=\footnotesize},
  leaf/.style={vertex, rectangle, rounded corners=0.6pt},
  tree edge/.style={draw=black!65, line width=0.55pt},
  branch edge/.style={
    tree edge,
    preaction={draw=yellow!85, line width=4pt, line cap=round}
  },
  rank label/.style={anchor=west, inner sep=1pt,
                     font=\scriptsize, text=blue!55!black}
]

  \node[heading] at (0,0) {Nodes, depth, rank, and branches};

  \node[note, font=\footnotesize\bfseries] at (0,-0.50) {Depth};
  \foreach \d in {0,1,2,3}
    \node[font=\footnotesize] at (0.40,{-0.85-0.86*\d}) {$\d$};

  \filldraw[fill=teal!7, draw=teal!55!black, rounded corners=4pt]
    (2.32,-2.20) rectangle (5.00,-3.82);
  \node[font=\footnotesize, text=teal!60!black]
    at (4.65,-2.72) {$W_s$};

  \node[vertex] (root) at (3.18,-0.85) {$r_W$};
  \node[vertex] (a)    at (2.22,-1.71) {$a$};
  \node[vertex] (b)    at (5.56,-1.71) {$b$};
  \node[leaf]   (c)    at (1.30,-2.57) {$c$};
  \node[vertex] (s)    at (3.40,-2.57) {$s$};
  \node[leaf]   (g)    at (5.56,-2.57) {$g$};
  \node[leaf]   (u)    at (2.85,-3.43) {$u$};
  \node[leaf]   (t)    at (4.22,-3.43) {$t$};

  \draw[tree edge] (root) -- (b) -- (g);
  \draw[tree edge] (a) -- (c);
  \draw[tree edge] (s) -- (u);
  \foreach \parent/\child in {root/a,a/s,s/t}
    \draw[branch edge] (\parent) -- (\child);

  \foreach \vname/\rk in {root/3,a/2,b/1,c/0,s/1,g/0,u/0,t/0}
    \node[rank label] at (\vname.east) {$(\rk)$};

  \node[note] at (6.45,-0.85)
    {The number in parentheses beside each node is its rank.};
  \node[note] at (6.45,-1.40)
    {The branch $r_W\prec a\prec s\prec t$ is highlighted in yellow.};
  \node[note] at (6.45,-1.95)
    {The root is $r_W$ and $\operatorname{Term}(W)=\{c,g,u,t\}$.};
  \node[note] at (6.45,-2.50)
    {$s$ is an immediate child of $a$.};
  \node[note] at (6.45,-3.05)
    {The shaded subtree is $W_s=\{v\in W:s\preceq v\}=\{s,u,t\}$.};
  \node[note] at (6.45,-3.60)
    {The terminal descendants of $s$ are $u$ and $t$.};
  \node[note] at (6.45,-4.15)
    {The leaf $t$ is its own terminal descendant.};

  \node[note] at (0,-4.20)
    {Circles are internal nodes. Squares are leaves.};
  \node[note] at (0,-4.70)
    {$\operatorname{depth}(t)=\#\{r_W,a,s\}=3$.};

\end{tikzpicture}%
}
\end{center}

A \emph{prefix-closed pruning} of \(W\) is a rooted
subtree \(W'\subset W\) such that every predecessor
in \(W\) of a node of \(W'\) also belongs to \(W'\).
A terminal descendant of \(s\) is a node in
\(W_s\cap\Term(W)\); in particular, a terminal node
is its own terminal descendant.

\begin{center}
\resizebox{0.98\linewidth}{!}{%
\begin{tikzpicture}[
x=1cm, y=1cm,
  heading/.style={anchor=west, font=\small\sffamily\bfseries},
  note/.style={anchor=west, align=left, font=\footnotesize, text=black!88},
  vertex/.style={circle, draw=black!70, fill=white, minimum size=6mm,
    inner sep=1pt, font=\footnotesize},
  leaf/.style={vertex, rectangle, rounded corners=0.6pt},
  point/.style={circle, draw=black!70, fill=white, inner sep=0pt,
    minimum size=3.5mm},
  endpoint/.style={point, rectangle},
  retained/.style={draw=blue!55!black, text=blue!55!black, fill=blue!5,
    rounded corners=1.5pt, inner xsep=5pt, inner ysep=3pt, font=\footnotesize},
  removed/.style={draw=black!48, dashed, text=black!60, fill=white,
    rounded corners=1.5pt, inner xsep=5pt, inner ysep=3pt, font=\footnotesize},
  edge/.style={draw=black!65, line width=0.55pt},
  path edge/.style={draw=blue!55!black, line width=1.15pt},
  removed edge/.style={draw=black!48, dashed, line width=0.65pt},
  divider/.style={draw=black!20, line width=0.45pt}
]
\useasboundingbox (0,-11.1) rectangle (14.6,-14.0);
  \node[heading] at (0,-11.37) {A prefix-closed pruning $W'\subset W$};
  \node[retained] (c-root) at (3.20,-12.00) {$\varnothing$};
  \node[retained] (c-a) at (2.10,-12.73) {$(1)$};
  \node[retained] (c-b) at (4.22,-12.73) {$(3)$};
  \node[removed] (c-c) at (5.61,-12.73) {$(5)$};
  \node[retained] (c-u) at (1.18,-13.51) {$(1,2)$};
  \node[removed] (c-v) at (2.94,-13.51) {$(1,4)$};
  \node[removed] (c-w) at (4.64,-13.51) {$(3,4)$};
  \draw[path edge] (c-root) -- (c-a) -- (c-u) (c-root) -- (c-b);
  \draw[removed edge] (c-root) -- (c-c) (c-a) -- (c-v) (c-b) -- (c-w);
  \node[note] at (6.35,-12.02) {Solid nodes are retained. Dashed nodes are removed.};
  \node[note] at (6.35,-12.54) {Keeping $(1,2)$ also keeps $(1)$ and $\varnothing$.};
  \node[note] at (6.35,-13.06) {In general, $t\in W',\ s\preceq t\ \Longrightarrow\ s\in W'$.};
  \node[note] at (6.35,-13.58) {The node $(3)$ becomes a leaf of $W'$.};
\end{tikzpicture}%
}
\end{center}

\rev{For any tree \(T\) and \(t\in T\), we write
\[
 \operatorname{Succ}_T(t)
 =
 \{u\in T:u\text{ is an immediate child of }t\}
\]
for the set of immediate children of \(t\) in \(T\) and 
\begin{equation*}
 \Term_T(t)=\{F\in\Term(T):t\preceq F\}.
\end{equation*}}

The following pruning lemma is a consequence
of 
\cite[Theorem~3.23]{Pequignot2017}.

\begin{lemma}\label{lem:reading}
Let $W$ be a rooted well-founded tree in which
every internal node has infinitely many immediate
children. For every nonroot node $t$, let
\[
 q_t:\Term_W(t)\longrightarrow Q_t
\]
be an arbitrary map into a finite set.
There is a prefix-closed pruning $W'\subset W$
containing the root of $W$ such that
\begin{enumerate}
\item if $t\in W'$ is internal in $W$, then
$\operatorname{Succ}_{W'}(t)$ is infinite;
\item $t\in W'$ is terminal if and only if
it is terminal in $W$;
\item for every nonroot node $t\in W'$,
the map $q_t$ is constant on $\Term_{W'}(t)$.
\end{enumerate}
\end{lemma}

\begin{center}
\resizebox{0.98\linewidth}{!}{%
\begin{tikzpicture}[
  heading/.style={anchor=west, font=\sffamily\bfseries},
  note/.style={anchor=west, font=\small},
  internal/.style={draw=black!70, fill=white, rounded corners=2pt,
    inner xsep=4pt, minimum height=5.6mm, font=\small},
  edge/.style={draw=black!65, line width=0.55pt},
  more/.style={edge, densely dotted},
  cell/.style={draw=black!65, line width=0.4pt}
]
  \colorlet{readingBlue}{blue!65!cyan}
  \colorlet{readingOrange}{orange!92!black}
  \colorlet{readingGreen}{teal!75!black}
  \colorlet{readingPurple}{violet!80!blue}
  \colorlet{readingRed}{red!75!black}

  \def\readingLeaf#1#2#3#4#5{%
    \coordinate (leaf) at #2;
    \draw[edge] (#1) -- (leaf);
    \begin{scope}[shift={(leaf)}]
      \foreach \cellColor [count=\cellIndex from 0] in {#5} {
        \path[cell, fill=\cellColor]
          ({0.24*(\cellIndex-#4/2)},-0.12)
          rectangle ({0.24*(\cellIndex+1-#4/2)},0.12);
      }
      \node[font=\small] at (0,-0.34) {$#3$};
    \end{scope}%
  }

  \def\readingDeepLeaf#1#2#3#4{%
    \coordinate (deep-leaf) at #2;
    \draw[edge] (#1) -- (deep-leaf);
    \begin{scope}[shift={(deep-leaf)}]
      \foreach \cellColor [count=\cellIndex from 0] in
        {readingBlue,readingBlue,readingOrange,readingGreen} {
        \path[cell, fill=\cellColor]
          ({-0.90+0.24*\cellIndex},-0.12)
          rectangle ({-0.90+0.24*(\cellIndex+1)},0.12);
      }
      \path[cell, fill=white] (0.06,-0.12) rectangle (0.66,0.12);
      \node[font=\scriptsize] at (0.36,0) {$\cdots$};
      \path[cell, fill=#4] (0.66,-0.12) rectangle (0.90,0.12);
      \node[font=\small] at (0,-0.34) {$#3$};
    \end{scope}%
  }

  \foreach \panel/\yOffset/\headingText/\secondIndex in
    {before/0/Before pruning/2, after/-7.50/After pruning/3} {
    \begin{scope}[yshift=\yOffset cm]
      \node[heading] at (0,0) {\headingText};
      \ifnum\secondIndex=3\relax
        \node[note] at (3.75,0)
          {Each retained internal node still has infinitely many children.};
      \else
        \node[note] at (3.75,0)
          {The dashed boxes highlight selected deletions.};
        \filldraw[fill=black!4, draw=black!65, densely dashed, rounded corners=3pt]
          (1.30,-2.56) rectangle (2.17,-3.25)
          (4.46,-2.56) rectangle (5.31,-3.25)
          (11.47,-3.32) rectangle (14.71,-5.02);
      \fi

      \node[note, font=\small\bfseries] at (0,-0.50) {Depth};
      \foreach \depth in {0,1,2,3,4} {
        \pgfmathsetmacro{\depthY}{-0.90-0.91*\depth}
        \node[note] at (0.30,\depthY) {$\depth$};
        \draw[black!12] (1.12,\depthY) -- (14.7,\depthY);
      }
      \node[note] at (0.30,-5.10) {$\vdots$};
      \node[note] at (0.12,-5.78) {$h-1$};
      \node[note] at (0.30,-6.55) {$h$};
      \draw[black!12] (1.12,-6.55) -- (14.7,-6.55);
      \node[internal] (root-\panel) at (6.86,-0.90) {$r_W$};
      \node[internal, line width=0.85pt] (a1-\panel) at (8.83,-1.81) {$a_1$};
      \node[internal] (a2-\panel) at (13.80,-1.81) {$a_2$};
      \node (more-root-\panel) at (5.04,-1.81) {$\cdots$};
      \draw[edge] (root-\panel) -- (a1-\panel) (root-\panel) -- (a2-\panel);
      \draw[more] (root-\panel) -- (more-root-\panel);
      \foreach \leafIndex/\leafX in {3/2.03,4/3.62} {
        \node[cell, circle, fill=white, inner sep=0pt, minimum size=2.4mm]
          (depth-one-leaf) at (\leafX,-1.81) {};
        \draw[edge] (root-\panel) -- (depth-one-leaf);
        \node[font=\small] at (\leafX,-2.15) {$a_{\leafIndex}$};
      }

      \readingLeaf{a1-\panel}{(3.36,-2.72)}{b_3}{1}{readingBlue}
      \readingLeaf{a2-\panel}{(13.54,-2.72)}{b_5}{1}{readingOrange}
      \ifnum\secondIndex=2\relax
        \readingLeaf{a1-\panel}{(1.74,-2.72)}{b_9}{1}{readingGreen}
        \readingLeaf{a1-\panel}{(4.85,-2.72)}{b_4}{1}{readingOrange}
      \else
        \node[font=\scriptsize, text=black!70] at (1.74,-2.93) {$b_9$ removed};
        \node[font=\scriptsize, text=black!70] at (4.85,-2.93) {$b_4$ removed};
      \fi
      \node (more-a2-\panel) at (14.76,-2.72) {$\cdots$};
      \draw[more] (a2-\panel) -- (more-a2-\panel);
      \node[internal] (b1-\panel) at (6.63,-2.72) {$b_1$};
      \node[internal] (b2-\panel) at (10.37,-2.72) {$b_2$};
      \node[font=\small] (more-a1-\panel) at (12.14,-2.72) {$b_6,b_8,\ldots$};
      \node[font=\scriptsize, anchor=south] at (11.94,-2.40)
        {Leaves at depths $5,6,\ldots$};
      \draw[edge] (a1-\panel) -- (b1-\panel) (a1-\panel) -- (b2-\panel);
      \draw[more] (a1-\panel) -- (more-a1-\panel);

      \node[internal] (c1-\panel) at (9.52,-3.63) {$c_1$};
      \node (more-b1-\panel) at (8.20,-3.63) {$\cdots$};
      \node (more-b2-\panel) at (11.05,-3.63) {$\cdots$};
      \node (more-c1-\panel) at (10.87,-4.54) {$\cdots$};
      \draw[edge] (b2-\panel) -- (c1-\panel);
      \draw[more] (b1-\panel) -- (more-b1-\panel)
        (b2-\panel) -- (more-b2-\panel) (c1-\panel) -- (more-c1-\panel);
      \readingLeaf{b1-\panel}{(5.65,-3.63)}{c_3}{2}{readingBlue,readingOrange}
      \readingLeaf{c1-\panel}{(8.54,-4.54)}{d_1}{3}{readingBlue,readingBlue,readingOrange}
      \node[internal] (d6-\panel) at (6.10,-4.54) {$d_6$};
      \draw[edge] (c1-\panel) -- (d6-\panel);
      \node (more-d6-\panel) at (8.10,-5.08) {$\cdots$};
      \draw[more] (d6-\panel) -- (more-d6-\panel);
      \node at (6.10,-5.10) {$\vdots$};
      \draw[more] (d6-\panel) -- (6.10,-4.89);
      \draw[more] (6.10,-5.34) -- (6.10,-5.50);
      \node[font=\scriptsize, anchor=west] at (7.10,-5.40)
        {Finitely many omitted levels};
      \node[internal] (p-\panel) at (6.10,-5.78) {$p$};
      \node[note, align=left] at (9.50,-5.90)
        {$h$ is the depth of the\\displayed leaves $F_i$.};
      \readingDeepLeaf{p-\panel}{(4.30,-6.55)}{F_1}{readingPurple}
      \node (more-p-\panel) at (9.80,-6.55) {$\cdots$};
      \draw[more] (p-\panel) -- (more-p-\panel);
      \ifnum\secondIndex=2\relax
        \readingDeepLeaf{p-\panel}{(7.90,-6.55)}{F_2}{readingRed}
        \readingLeaf{b1-\panel}{(6.92,-3.63)}{c_4}{2}{readingOrange,readingBlue}
        \readingLeaf{c1-\panel}{(9.84,-4.54)}{d_2}{3}{readingOrange,readingOrange,readingBlue}
        \node[internal] (c2-\panel) at (12.62,-3.63) {$c_2$};
        \node (more-c2-\panel) at (14.20,-4.54) {$\cdots$};
        \draw[edge] (b2-\panel) -- (c2-\panel);
        \draw[more] (c2-\panel) -- (more-c2-\panel);
        \readingLeaf{c2-\panel}{(11.98,-4.54)}{d_4}{3}
          {readingOrange,readingOrange,readingOrange}
        \readingLeaf{c2-\panel}{(13.23,-4.54)}{d_5}{3}
          {readingBlue,readingBlue,readingBlue}
      \else
        \readingDeepLeaf{p-\panel}{(7.90,-6.55)}{F_3}{readingPurple}
        \readingLeaf{b1-\panel}{(6.92,-3.63)}{c_5}{2}{readingBlue,readingOrange}
        \readingLeaf{c1-\panel}{(9.84,-4.54)}{d_3}{3}{readingBlue,readingBlue,readingOrange}
        \node[note, align=left, text=black!70] at (11.47,-4.11)
          {The subtree rooted\\at $c_2$ is removed.};
      \fi
    \end{scope}
  }
  \draw[black!25] (0,-7.15) -- (14.7,-7.15);

  \node[note, text width=14.5cm, align=left] at (0,-15.00)
    {At each leaf $F$, colored cells show $q_t(F)$ for its proper nonroot ancestors $t$,
     ordered from left to right by increasing depth, i.e.,  $(q_{a_1}(d_1), q_{b_2}(d_1), q_{c_1}(d_1)$).};
\end{tikzpicture}%
}
\end{center}
\begin{proof}
We observe that every node has a terminal descendant; otherwise one could
choose successive immediate children indefinitely and obtain an infinite
branch.

For a node \(s\in W\) with  depth \(k\), write its unique predecessor chain as
\[
 r_W=s_0\prec s_1\prec\cdots\prec s_k=s.
\]
Define the map
\[
 a_s(F)=
 \begin{cases}
  \bigl(q_{s_1}(F),\ldots,q_{s_{k-1}}(F)\bigr),& k>1,\\
  (),& k\leq1,
 \end{cases}
 \qquad F\in\Term(W_s),
\]
where \(()\) denotes the empty tuple.  Since \(Q_t\) is finite for every \(t\), the range of \(a_s\) is finite.

We prove by transfinite induction on the node rank \(\rho_W(s)\) that there
exists a prefix-closed pruning \(U^s\subset W_s\) satisfying conditions \emph{(1)} and
\emph{(2)}, such that
\[
 a_s(F)=a_s(G)\qquad\bigl(\forall F,G\in\Term(U^s)\bigr),
\]
and such that, for every nonroot node \(u\in U^s\), the restriction
\[
 q_u\big|_{\{F\in\Term(U^s):u\preceq F\}}
\]
is constant.
If \(s\) is terminal, take \(U^s=\{s\}\); the assertion is then immediate.


Suppose that \(s\) is internal. For each immediate child \(t\) of \(s\),
apply the induction hypothesis to obtain a pruning \(U^t\) rooted at
\(t\). Let \(b_t\) denote the constant value of \(a_t\) on
\(\Term(U^t)\). 


All the labels \(b_t\) belong to the same finite product
\(\prod_{j=1}^k Q_{s_j}\), interpreted as a empty product when \(s\) is
the root, i.e., \(q_{s_j}(\Term_W({s_j}))\subset Q_{s_j}\). Since \(s\) has infinitely many immediate children, some
label \(b\) occurs infinitely often. Let
\[
 C_s=\{t:t\in \Succ_W(s)\text{ and }b_t=b\}.
\]
Put
\[
 U^s=\{s\}\cup\bigcup_{t\in C_s}U^t.
\]
When \(s\neq r_W\), we have that 
\[
\Term(U^s)=\bigcup_{t\in C_s}\Term(U^t),
\]
and
\[
 a_t(F)=\bigl(a_s(F),q_s(F)\bigr)=b,
\]
for every \(t\in C_s\) and
\(F\in\Term(U^t)\). Then, we have that \(a_s\) and \(q_s\) are constant on \(\Term(U^s)\).  When \(s= r_W\),
\(a_s\) is the empty tuple and hence is constant.  Hence, \(U^s\) satisfies the
induction assertion.  By inducting on the node rank, we obtain the desired pruning. This completes the proof.
\end{proof}

\section{Proofs of the Theorems \ref{thm:transfer} and  \ref{thm:main}}\label{sec:finite}
With the auxiliary lemma established in Section \ref{sec:tree-prelim}, we prove Theorem \ref{thm:transfer} below.


\begin{proof}[Proof of Theorem \ref{thm:transfer}]
Pass first to a basic subsequence and relabel it as $(x_n)$. Fix $B$ in
\eqref{eq:coefficient-bound} and put 
\begin{align}\label{eq:H=sup}
H=\sup_n\|x_n\|_X.
\end{align}
\medskip
\noindent\textbf{Step 1.}
For every nonempty finite set $F=\{n_1<\cdots<n_m\}\subset\mathbb N$, put
\begin{equation*}
 \lambda(F)=\inf\left\{
 \left\|\sum_{k=1}^m a_kx_{n_k}\right\|_X:
 \begin{array}{l}
  a_1,\ldots,a_m\in\mathbb C,\quad \gamma_k\in\{-1,1\},\\
  \left\|J\sum_{k=1}^m\gamma_ka_kx_{n_k}\right\|_Y=1
 \end{array}\right\},
 \qquad \lambda(\varnothing)=+\infty.
\end{equation*}
We regard finite strictly increasing sequences as nodes,
ordered by initial segments (a node
\(s=(n_1,\ldots,n_j)\) precedes \(t=(m_1,\ldots,m_k)\) when
\(j\leq k\) and \(n_i=m_i\) for every \(1\leq i\leq j\)), with $\varnothing$ as the root.
Fix $0<c<C^{-1}$ and define
\begin{equation*}
 T_c=\Bigl\{(n_1<\cdots<n_m):
 \lambda\bigl(\{n_1,\ldots,n_m\}\bigr)\ge c\Bigr\}
 \cup\{\varnothing\}.
\end{equation*}
Since $\lambda(F)\ge\lambda(G)$ whenever $F\subseteq G$,
$T_c$ is a rooted tree ordered by initial segments. We call the nodes of \(T_c\) \(c\)-good.

{\it Assume by contradiction that \((x_n)\) has no subsequence
satisfying \eqref{eq:transfer-conclusion}.}
Then $T_c$ is well founded.
Indeed,  otherwise, \(T_c\) would contain an infinite branch,  every node of which satisfies \(\lambda\geq c\). Then, this infinite branch yields a subsequence of $(x_n)$ satisfying \eqref{eq:transfer-conclusion}, contradicting our assumption.

For a nonempty \(c\)-good node \(s=(n_1,\ldots,n_k)\), write
\(\max s=n_k\) and \(s^\frown n=(n_1,\ldots,n_k,n)\).  A
\emph{first-\(c\)-bad node} is a non-\(c\)-good node whose immediate predecessor is \(c\)-good.
Complete \(T_c\)
to its first-\(c\)-bad terminal tree \(\widehat T_c\): every \(c\)-good node \(s\)
has all children \(s^\frown n\), \(n>\max s\); at the root, every singleton \((n)\), \(n\in\mathbb N\), is allowed as an immediate child. We observe that \(\widehat T_c\) is still well
founded and each internal node has infinitely many immediate children. Indeed, every \(c\)-good node of \(\widehat T_c\) belongs to \(T_c\), whereas each newly adjoined node is a first-\(c\)-bad node and is therefore terminal.
Moreover, 
\[
 \lambda(F)<c\qquad\bigl(F\in\Term(\widehat T_c)\bigr).
\]
Consequently, for each terminal node $F=(n_1,\ldots,n_m)$ there exist
$b^F=\bigl(b_k^F\bigr)_{k=1}^m\in\mathbb C^m$ and
$\bigl(\gamma_k^F\bigr)_{k=1}^m\in\{-1,1\}^m$ such that
\begin{equation}\label{eq:bad-witness}
 \begin{gathered}
\left\|J\sum_{k=1}^m\gamma_k^Fb_k^Fx_{n_k}\right\|_Y=1,\qquad\left\|\sum_{k=1}^m b_k^Fx_{n_k}\right\|_X<c.
 \end{gathered}
\end{equation}
Since $(x_n)$ is a basic sequence  and weakly null, it follows from \eqref{eq:coefficient-bound} that
\begin{equation*}
 \bigl|b_k^F\bigr|\overset{\eqref{eq:coefficient-bound}}{\le} B\left\|\sum_{k=1}^m b_k^Fx_{n_k}\right\|_X\overset{\eqref{eq:bad-witness}}{<}Bc.
\end{equation*}
Choose positive summable sequences $(\theta_k)$ and $(\eta_k)$,
whose sums can be made arbitrarily small.
For each $k\ge1$, partition the scalar disk
$\{a\in\mathbb C:|a|\le Bc\}$ into finitely many sets of diameter at most
$\theta_k$. 
For each nonroot node $t$ of depth $k$, define
\[
q_t:\Term_{\widehat T_c}(t)
\longrightarrow\{1,\ldots,N_k\}\times\{-1,1\},\qquad q_t(F):=\bigl(j,\gamma_k^F\bigr), \qquad F\in \Term_{\widehat T_c}(t)
\]
whenever $b_k^F\in C_{k,j}$, where
$\{C_{k,j}\}_{j=1}^{N_k}$ is the finite partition
defined above. 
Then, by Lemma~\ref{lem:reading},
we obtain a prefix-closed, infinitely splitting, pruning $W_1\subset\widehat T_c$
such that $q_t$ is constant on the terminal
descendants of every nonroot node $t\in W_1$.
In other words, Lemma~\ref{lem:reading} yields a
pruning $W_1\subset\widehat T_c$ and labels
$\alpha_t\in\mathbb C$, $\gamma_t\in\{-1,1\}$ satisfying
\begin{equation}\label{eq:stable-current}
 |\alpha_t|\le Bc,\qquad
 \bigl|\alpha_t-b_k^F\bigr|\le\theta_k,\qquad
 \gamma_t=\gamma_k^F
 \quad\bigl(t\in W_1\text{ with depth }k,\ F\in\Term_{W_1}(t)\bigr).
\end{equation}
For a node $F=(n_1,\ldots,n_m)$, write
$$F|0=\varnothing,\qquad F|k=(n_1,\ldots,n_k),\qquad 1\le k\le m$$

We construct a new tree $W\subset W_1$ as follows.
Starting from the root, at each node $s$ internal in $W_1$,
retain infinitely many immediate children
$t$ with the same sign $\gamma_t\in\{-1,1\}$,
and denote their common sign by $\xi_s$.
The tree $W$ is prefix-closed and infinitely
splitting, with $\Term(W)\subset\Term(W_1)$.
Moreover, \eqref{eq:stable-current} remains valid,
and
\begin{equation}\label{eq:predictable-labels}
 \gamma_k^F=\gamma_{F|k}=\xi_{F|k-1}
 \qquad
 \bigl(F=(n_1,\ldots,n_m)\in\Term(W),\ 1\le k\le m\bigr).
\end{equation}
Thus the $k$-th sign depends only on the
preceding node $F|k-1$.

\medskip
\noindent\textbf{Step 2.}
Let $s\in W$ be internal with depth $k-1$. Its children have the form
$t=s^\frown n$. Since $(\alpha_t)$ is bounded and $x_n\weakto0$,
\begin{equation*}
 \alpha_{s^\frown n}x_n\weakto0
 \quad\text{as }n\to\infty..
\end{equation*}
By Mazur's lemma \cite[Corollary~3.8]{Brezis2011}, there exist a finite nonempty
set $A_s\subset \Succ_W(s)$ and positive weights $\bigl(\pi_s(t)\bigr)_{t\in A_s}$
satisfying
\begin{equation}\label{eq:mazur-node}
 \sum_{t\in A_s}\pi_s(t)=1,
 \qquad
 \left\|\sum_{\substack{t=s^\frown n\\t\in A_s}}
       \pi_s(t)\alpha_tx_n\right\|_X<\eta_k.
\end{equation}
Then, we construct a subtree $P\subset W$ recursively as follows.
Include the root of $W$, and whenever $s\in P$ is internal in $W$,
retain precisely the immediate children in $A_s$. Thus $P$ is finitely branching and well founded.
By K\"onig's lemma \cite[Lemma~8.1.2]{DiestelGraph2017},
$P$ is finite. Moreover, \(\Term(P)\subset\Term(W)\).

Let $\Omega=\Term(P)$, and let $N$ be the maximum depth of its terminal nodes. For a
terminal node $F=(n_1,\ldots,n_m)$, define
\begin{equation}\label{eq:path-probability}
 \PP\bigl(\{F\}\bigr)=\prod_{k=1}^m\pi_{F|k-1}(F|k).
\end{equation}
By \eqref{eq:mazur-node}, these probabilities
sum to one. For a terminal node \(F\) of depth $m$, put
\begin{equation*}
 Z_k(F)=F|\min\{k,m\},\qquad \F_k=\sigma(Z_k)
 \qquad(0\le k\le N).
\end{equation*}
For an internal node $s\in P$ of depth $k-1$,
\begin{equation}\label{eq:transition}
 \PP(Z_k=t\mid Z_{k-1}=s)\overset{\eqref{eq:mazur-node},\eqref{eq:path-probability}}{=}\pi_s(t),\qquad t\in A_s.
\end{equation}

For a terminal node \(F\) of depth $m$, define
\begin{equation}\label{eq:rmd}
 \begin{split}
 r_k(F)&=\begin{cases}
 \alpha_{F|k}x_{n_k},&k\le m,\\
 0,&k>m,
 \end{cases}\\
 m_k&=\E(r_k\mid\F_{k-1}),\qquad d_k=r_k-m_k.
 \end{split}
\end{equation}
Then $(d_k)_{k=1}^N$ is an $X$-valued martingale difference sequence.
If $s=Z_{k-1}(F)$ is internal in $P$, then
\begin{align*}
 \bigl\|m_k(F)\bigr\|_X
 &\overset{\eqref{eq:transition}}{=}
 \left\|\sum_{\substack{t=s^\frown n\\t\in A_s}}
          \pi_s(t)\alpha_tx_n\right\|_X
 \overset{\eqref{eq:mazur-node}}{<}\eta_k.
\end{align*}
If $s$ is a terminal node of $P$, then $m_k(F)=0$. Hence
\begin{equation}\label{eq:mean-bound}
 \sup_{F\in\Omega}\bigl\|m_k(F)\bigr\|_X\le\eta_k.
\end{equation}
For each $F=(n_1,\ldots,n_m)\in\Omega$, define
\begin{equation*}
 \xi_{k-1}(F)=
 \begin{cases}
 \xi_{F|k-1},&k\le m,\\
 1,&k>m.
 \end{cases}
\end{equation*}
Since $\xi_{k-1}$ depends only on $Z_{k-1}$,
it is $\mathcal F_{k-1}$-measurable.
Hence, $(\xi_{k-1})_{k=1}^N$ is a
predictable sign sequence.


\noindent\textbf{Step 3.}
For each $F=(n_1,\ldots,n_m)\in\Omega$, we have that
\begin{eqnarray}
\left\|\sum_{k=1}^N d_k(F)\right\|_X
&\overset{\eqref{eq:rmd}}{=}&
\left\|
\sum_{k=1}^m\alpha_{F|k}x_{n_k}
-\sum_{k=1}^N m_k(F)
\right\|_X
\nonumber\\
&\le&
\left\|\sum_{k=1}^m b_k^Fx_{n_k}\right\|_X
+\sum_{k=1}^m
\bigl|\alpha_{F|k}-b_k^F\bigr|\|x_{n_k}\|_X
+\sum_{k=1}^N\bigl\|m_k(F)\bigr\|_X
\nonumber\\
&\overset{\eqref{eq:H=sup},\eqref{eq:bad-witness},
\eqref{eq:stable-current},
\eqref{eq:mean-bound}}{\le}&
c+H\sum_{k\ge1}\theta_k
+\sum_{k\ge1}\eta_k.
\label{eq:ordinary-sum-bound}
\end{eqnarray}
Similarly,
\begin{eqnarray}
\left\|
\sum_{k=1}^N\xi_{k-1}(F)d_k(F)
-\sum_{k=1}^m\gamma_k^Fb_k^Fx_{n_k}
\right\|_X
&\overset{\eqref{eq:rmd},
\eqref{eq:predictable-labels}}{=}&
\left\|
\sum_{k=1}^m
\gamma_k^F\bigl(\alpha_{F|k}-b_k^F\bigr)x_{n_k}
-\sum_{k=1}^N\xi_{k-1}(F)m_k(F)
\right\|_X
\nonumber\\
&\le&
\sum_{k=1}^m
\bigl|\alpha_{F|k}-b_k^F\bigr|\|x_{n_k}\|_X
+\sum_{k=1}^N\bigl\|m_k(F)\bigr\|_X
\nonumber\\
&\overset{\eqref{eq:H=sup},\eqref{eq:stable-current},
\eqref{eq:mean-bound}}{\le}&
H\sum_{k\ge1}\theta_k
+\sum_{k\ge1}\eta_k.
\label{eq:transformed-sum-bound}
\end{eqnarray}
By the $r$-triangle inequality \eqref{eq:r-triangle},
\begin{eqnarray}
1
&\overset{\eqref{eq:bad-witness}}{=}&
\left\|
J\sum_{k=1}^m\gamma_k^Fb_k^Fx_{n_k}
\right\|_Y^r
\nonumber\\
&\overset{\eqref{eq:r-triangle}}{\le}&
\left\|
J\sum_{k=1}^N\xi_{k-1}(F)d_k(F)
\right\|_Y^r
+
\left\|
J\left(
\sum_{k=1}^N\xi_{k-1}(F)d_k(F)
-\sum_{k=1}^m\gamma_k^Fb_k^Fx_{n_k}
\right)
\right\|_Y^r
\nonumber\\
&\overset{\eqref{eq:transformed-sum-bound}}{\le}&
\left\|
J\sum_{k=1}^N\xi_{k-1}(F)d_k(F)
\right\|_Y^r
+\|J\|^r
\left(
H\sum_{k\ge1}\theta_k+
\sum_{k\ge1}\eta_k
\right)^r.
\label{eq:path-lower-bound}
\end{eqnarray}
Taking expectations and applying \eqref{eq:martingale-input}, we obtain that
\begin{eqnarray}
1-\|J\|^r
\left(
H\sum_{k\ge1}\theta_k+
\sum_{k\ge1}\eta_k
\right)^r
&\overset{\eqref{eq:path-lower-bound}}{\le}&
\E\left\|
J\sum_{k=1}^N\xi_{k-1}d_k
\right\|_Y^r
\overset{\eqref{eq:martingale-input}}{\le}
C^r\left(
\E\left\|\sum_{k=1}^N d_k\right\|_X^s
\right)^{r/s}
\nonumber\\
&\overset{\eqref{eq:ordinary-sum-bound}}{\le}&
C^r\left(
c+H\sum_{k\ge1}\theta_k+
\sum_{k\ge1}\eta_k
\right)^r.
\label{eq:master-contradiction}
\end{eqnarray}
Since the sums of
$(\theta_k)$ and $(\eta_k)$ can be chosen
arbitrarily small, \eqref{eq:master-contradiction}
gives a contradiction whenever $0<c<C^{-1}$.
Hence, $T_c$ has an infinite branch for every
$0<c<C^{-1}$, and \eqref{eq:transfer-conclusion}
holds with any $C'=c^{-1}>C$.
\end{proof}

We next apply the preceding theorem to noncommutative
$L_1$-spaces associated with finite von Neumann algebras
to obtain the following result, which will be used to
establish the existence of $2$-co-lacunary subsequences
in finite von Neumann algebras.

\begin{corollary}
\label{cor:Lr}
Let \(\M\) be a von Neumann algebra equipped with a normal faithful
tracial state \(\tau\). For $0<r<1$,
there exists a constant $A_r<\infty$ such that
every seminormalized weakly null sequence in
$L_1(\M,\tau)$ has a basic subsequence $(u_k)$
satisfying
\begin{equation*}
 \left\|\sum_k\gamma_ka_ku_k\right\|_r
 \le A_r\left\|\sum_ka_ku_k\right\|_1
\end{equation*}
for all finitely supported complex scalar sequences and  all signs
$\gamma_k\in\{-1,1\}$.
\end{corollary}

\begin{proof}
Let $(d_k)_{k=1}^m$ be a finite
$L_1(\M,\tau)$-valued martingale difference
sequence on a filtered probability space
$\bigl(\Omega,\mathcal F,(\mathcal F_k)_{k=0}^m,\mathbb P\bigr)$,
where
\[
 \mathcal F_0\subseteq\mathcal F_1
 \subseteq\cdots\subseteq\mathcal F_m
 \subseteq\mathcal F,
\]
and let $(\xi_{k-1})$ be predictable signs.

Consider the product algebra
\[
 \N=L_\infty(\Omega)\,\overline\otimes\,\M,
 \qquad
 \widetilde\tau=\E_\Omega\otimes\tau,
\]
with the filtration
\[
 \N_k=L_\infty(\Omega,\mathcal F_k)
       \,\overline\otimes\,\M.
\]
\rev{Since $(\xi_{k-1})$ is predictable, it follows that
\[
 \xi_{k-1}\in L_\infty(\Omega,\mathcal F_{k-1}),
 \qquad
 \xi_{k-1}\otimes\one
 \in\N_{k-1}\cap Z(\N)
 \subseteq\N_{k-1}\cap\N_k'.
\]}
Then, it follows from
\eqref{eq:weak-embed} and \eqref{eq:MT} that
\begin{eqnarray}
\E_\Omega
 \left\|\sum_{k=1}^m\xi_{k-1}d_k\right\|_r^r
&=&
 \left\|\sum_{k=1}^m\xi_{k-1}d_k
 \right\|_{L_r(\N,\widetilde\tau)}^r
\overset{\eqref{eq:weak-embed}}{\le}
 \frac{1}{1-r}
 \left\|\sum_{k=1}^m\xi_{k-1}d_k
 \right\|_{L_{1,\infty}(\N,\widetilde\tau)}^r
\nonumber\\
&\overset{\eqref{eq:MT}}{\le}&
 \frac{C_{\mathrm{MT}}^r}{1-r}
 \left(
 \E_\Omega
 \left\|\sum_{k=1}^m d_k\right\|_1
 \right)^r.
 \label{eq:Lr-martingale-input}
\end{eqnarray}
By \eqref{eq:L1-Lp},
the natural inclusion
\[
 J:L_1(\M,\tau)\longrightarrow L_r(\M,\tau)
\]
is bounded.
Thus \eqref{eq:Lr-martingale-input} verifies
the hypothesis of Theorem~\ref{thm:transfer}
with $X=L_1(\M,\tau)$, $Y=L_r(\M,\tau)$
and $s=1$. Then, the conclusion follows from Theorem~\ref{thm:transfer}.
\end{proof}

\begin{corollary}\label{cor:finite-weak-null}
Let $\M$ carry a normal faithful finite trace $\tau$.
Every seminormalized weakly null sequence in
$L_1(\M,\tau)$ has a $2$-co-lacunary subsequence.
\end{corollary}

\begin{proof}
By replacing $\tau$ with $\tau/\tau(\one)$,
we may assume that $\tau(\one)=1$.

Let $(x_n)$ be a seminormalized weakly null
sequence in $L_1(\M,\tau)$, and put
$K=\{x_n:n\ge1\}$. Then
\begin{equation}\label{eq:K-seminorm}
 a:=\inf_{x\in K}\|x\|_1>0.
\end{equation}
Since $x_n\weakto0$, the set $K$ is
relatively weakly compact. By
\eqref{eq:equi-spectral} in
Proposition~\ref{prop:equi}, there exists
$R>0$ such that
\begin{equation}\label{eq:K-spectral}
 \sup_{x\in K}
 \tau\bigl(|x|\one_{(R,\infty)}(|x|)\bigr)
 <a/2.
\end{equation}
Fix $0<r<1$. For every $x\in K$,
we have
\begin{eqnarray*}
\|x\|_r^r
&\ge&
\tau\bigl(|x|^r\one_{[0,R]}(|x|)\bigr)
\ge
R^{r-1}
\tau\bigl(|x|\one_{[0,R]}(|x|)\bigr)\\
&\ge& R^{r-1}\Bigl(\tau(|x|)-\tau\bigl(|x|\one_{(R,\infty)}(|x|)\bigr)\Bigr)
\overset{\eqref{eq:K-seminorm},
\eqref{eq:K-spectral}}{\ge}
R^{r-1}a/2.
\label{eq:K-Lr-lower}
\end{eqnarray*}
Consequently,
\begin{equation}\label{eq:delta-r}
 \delta_r:=\inf_{n\ge1}\|x_n\|_r>0.
\end{equation}

By Corollary~\ref{cor:Lr}, there exist
a basic subsequence $(u_k)$ of $(x_n)$
and a constant $A_r<\infty$ such that
\begin{equation}\label{eq:finite-sign-estimate}
 \left\|\sum_k\gamma_ka_ku_k\right\|_r
 \le A_r\left\|\sum_ka_ku_k\right\|_1
\end{equation}
for all finitely supported complex scalar
sequences $(a_k)$ and all signs
$\gamma_k\in\{-1,1\}$.

Let $(\varepsilon_k)$ be independent
symmetric signs. By Lemma~\ref{lem:cotype}
and \eqref{eq:finite-sign-estimate},
\begin{eqnarray}
\left(\sum_k|a_k|^2\|u_k\|_r^2\right)^{1/2}
\overset{\eqref{eq:cotype}}{\le}
K_r\left(
\E_\varepsilon
\left\|\sum_k\varepsilon_ka_ku_k\right\|_r^r
\right)^{1/r}
\overset{\eqref{eq:finite-sign-estimate}}{\le}
K_rA_r\left\|\sum_ka_ku_k\right\|_1.
\label{eq:finite-colac-estimate}
\end{eqnarray}
Since $(u_k)$ is a subsequence of $(x_n)$, it follows from
\eqref{eq:delta-r} that
\begin{equation*}
 \left(\sum_k|a_k|^2\right)^{1/2}
 \overset{\eqref{eq:delta-r},
 \eqref{eq:finite-colac-estimate}}{\le}
 \frac{K_rA_r}{\delta_r}
 \left\|\sum_ka_ku_k\right\|_1.
\end{equation*}
Thus $(u_k)$ is $2$-co-lacunary. This completes the proof.
\end{proof}

In the setting of \(\sigma\)-finite von Neumann algebras, combining \lemref{lem:uniform-symmetric-factorization}
and \propref{prop:core-big-window} with
\corref{cor:finite-weak-null}, we obtain the following
result for seminormalized weakly null, $1$-biequiintegrable sequences in \(L_1(\M)\).

\begin{corollary}
\label{cor:biequi-weak-null}
Let \(\M\) be an arbitrary von Neumann algebra.  Every
seminormalized weakly null,
$1$-biequiintegrable sequence in \(L_1(\M)\) has a \(2\)-co-lacunary
subsequence.
\end{corollary}

\begin{proof}
It suffices to consider the case when $\M$ is $\sigma$-finite (see e.g. \cite[Lemma~4.3]{RandrianantoaninaHaagerup2002}). 
Let 
\begin{align}
\delta=\inf_n\|x_n\|_1>0, \qquad \varepsilon=\delta/8\label{eq:delta-inf-norm}
\end{align} 
Put
\(K=\{x_n:n\geq1\}\). By \lemref{lem:uniform-symmetric-factorization}, there exists
a constant $C_\varepsilon<\infty$ such that, for every $n$,
we can choose $m_n\in\M$ with
\begin{align}
 y_n:=D^{1/2}m_nD^{1/2},\qquad
 \|m_n\|_\infty\leq C_\varepsilon,\qquad
 \|x_n-y_n\|_1<\varepsilon.
 \label{eq:x_n-y_n}
\end{align}
Choose a positive integer \(N\) such that
\begin{align}
 N\geq\frac{2B_0C_\varepsilon}{\delta}.\label{eq:N-ge}
\end{align}
Then, define
\[
 \mathcal J_N:L_1(\M)\longrightarrow L_1(P_N\Core P_N,\Trc),
 \qquad
 \mathcal J_N(x):=\frac1{2N}P_NxP_N,
\]
and 
\[
T_N:\mathcal C\to\mathcal M,\qquad
T_N(z)=\frac1{2N}\Phi(P_NzP_N),
\]
where $P_N,\Core,\Trc,\Phi$ were defined in Section \ref{Haagerup}.
\rev{Applying  \cite[Lemma~3.3]{HaagerupKraus1994} to
\((2N)^{-1/2}P_N\)}, we have that 
\begin{align}\label{T_N}
\|T_N\|\le1,
\end{align}
where \(\|\cdot\|\) denotes the operator norm.
Then 
\[
\bigl\|\mathcal J_N(x)\bigr\|_{L_1(\Core,\Trc)}=\sup_{\substack{z\in\mathcal C\\\|z\|_\infty\le1}}
\Bigl|\operatorname{Tr}\bigl(\mathcal J_N(x)z\bigr)\Bigr|
\stackrel{\eqref{eq:window-weight-duality}}{=}
\sup_{\substack{z\in\mathcal C\\\|z\|_\infty\le1}}
\Bigl|\bigl\langle x,T_N(z)\bigr\rangle\Bigr|
\stackrel{\eqref{T_N}}{\le} \|x\|_{1},
\]
which implies that \(\mathcal J_N\) is a contraction. Then it follows from \eqref{eq:x_n-y_n} that
\begin{align}
    \bigl\|\mathcal J_N(x_n-y_n)\bigr\|_{L_1(\Core,\Trc)}\le \varepsilon\label{eq:J_N(x_n-y_n)}
\end{align}
By  \propref{prop:core-big-window}, for every \(n\), we
have that
\begin{eqnarray*}
 \|\mathcal J_Nx_n\|_{L_1(\Core,\Trc)}
 &\geq&
 \|\mathcal J_Ny_n\|_{L_1(\Core,\Trc)}
 -\bigl\|\mathcal J_N(x_n-y_n)\bigr\|_{L_1(\Core,\Trc)}\\
 &\stackrel{\eqref{eq:J_N(x_n-y_n)}}{\geq}&
 \|\mathcal J_Ny_n\|_{L_1(\Core,\Trc)}
 -\varepsilon
 \\
 &\stackrel{\eqref{eq:big-window-lower}}{\geq}&
 \|y_n\|_{1}
 -\frac{B_0C_\varepsilon}{2N}
 -\varepsilon\\
 &\stackrel{\eqref{eq:x_n-y_n}}{\geq}&
 \|x_n\|_{1}
 -2\varepsilon
 -\frac{B_0C_\varepsilon}{2N}
 \stackrel{\eqref{eq:delta-inf-norm},\eqref{eq:N-ge}}{\ge} \frac{\delta}{2},
\end{eqnarray*}
which implies that
\begin{align}
 \|\mathcal J_Nx_n\|_{L_1(\Core,\Trc)}
 \geq\frac\delta2.
 \label{eq:JN-seminormalized}
\end{align}
Moreover, summing
\eqref{eq:pj-trace} over \(-N\leq j\leq N-1\) gives
\[
 \Trc(P_N)=e^N-e^{-N}<\infty,
\]
which implies that the corner \(P_N\Core P_N\) has finite trace.
Since \(\mathcal J_N\) is bounded and linear, it follows from \(x_n\weakto0\) that 
\[\mathcal J_Nx_n\weakto0.\] 
Together with \eqref{eq:JN-seminormalized}, \((\mathcal{J}_Nx_n)\) satisfies the hypotheses of
\corref{cor:finite-weak-null}, and then \((\mathcal{J}_Nx_N)\) has a \(2\)-co-lacunary subsequence
\((\mathcal J_Nx_{n_k})\) with lower \(\ell_2\)-constant \(\gamma>0\). Then, it follows from the contractivity of $\mathcal J_N$ that
\[
 \left\|\sum_k a_kx_{n_k}\right\|_1
 \geq\left\|\sum_k a_k\mathcal J_Nx_{n_k}\right\|_{L_1(\Core,\Trc)}
 \geq\gamma\left(\sum_k|a_k|^2\right)^{1/2}.
\]
This completes the proof.
\end{proof}



The following lemma establishes a lower $\ell_2$-estimate
when either $\|x_np_n\|_1$ or $\|p_nx_n\|_1$
is uniformly bounded below for pairwise
disjoint projections $(p_n)$.

\begin{lemma}\label{lem:one-sided-witness}
Let \(\M\) be an arbitrary von Neumann algebra.
Let \((x_n)\subset L_1(\M)\) be weakly null and
let \((p_n)\subset\M\) be pairwise disjoint projections.
If, for some \(\delta>0\),
\begin{align}
  \inf_n\|x_np_n\|_1\geq\delta,
  \label{eq:right-mass}
\end{align}
then \((x_n)\) has a \(2\)-co-lacunary subsequence.
The same conclusion holds if \eqref{eq:right-mass}
is replaced by
\(\inf_n\|p_nx_n\|_1\geq\delta\).
\end{lemma}

\begin{proof}
We prove the right-support version; the left-support version follows by taking adjoints.

 Let
\begin{align}
 z_n=x_np_n=v_n|z_n|
  \label{eq:z_n=x_np_n}
\end{align}
be the polar decomposition. 
It follows from \(v_n^*v_n\leq p_n\) that 
\begin{align}
p_nv_n^*=v_n^*.\label{w_n=p_nw_n}
\end{align}  
Then we have that
\begin{eqnarray}
 \langle x_n,v_n^*\rangle
 &\overset{\eqref{eq:haagerup-pairing},\eqref{w_n=p_nw_n}}{=}&\operatorname{tr}_{\M}(x_np_nv_n^*)
 \overset{\eqref{eq:z_n=x_np_n}}{=}\operatorname{tr}_{\M}(v_n|z_n|v_n^*)\nonumber\\
 &=&\operatorname{tr}_{\M}(|z_n|v_n^*v_n)
 =\operatorname{tr}_{\M}(|z_n|)
  \overset{\eqref{eq:z_n=x_np_n}}{=}\|x_np_n\|_1,
 \label{eq:witness-diagonal}
\end{eqnarray}
and
\begin{align}\label{eq:vivj}
v_iv_j^*=v_ip_ip_jv_j^*=0\quad(i\ne j).
\end{align}
Since the projections $p_n$ are pairwise disjoint,
$v_n^*\to0$ ultraweakly. Together with $x_n\weakto0$,
we can recursively choose a subsequence, still denoted
by $(x_n)$, such that
\begin{align}\label{eq:HS-off-diagonal}
  \sum_{i\ne j}\Bigl|\bigl\langle x_i,v_j^*\bigr\rangle\Bigr|^2
  <\frac{\delta^2}{4}.
\end{align}

Let $(a_i)$ be a nonzero finitely supported scalar
sequence and set
\begin{align}\label{eq:r,w}
  r=\left(\sum_i|a_i|^2\right)^{1/2},
  \qquad
  w=\frac1r\sum_j\overline{a_j}v_j^*.
\end{align}
It follows from \eqref{eq:vivj} that
\begin{align}\label{eq:w_infty}
  \|w\|_\infty^2
  =\frac1{r^2}
   \left\|\sum_j|a_j|^2v_jv_j^*\right\|_\infty
  \le1.
\end{align}
Consequently, by \eqref{tr le 1 infty} and the Cauchy--Schwarz
inequality,
\begin{eqnarray*}
  \left\|\sum_i a_ix_i\right\|_1
  &\stackrel{\eqref{tr le 1 infty},\eqref{eq:w_infty}}{\ge}&
  \left|\left\langle\sum_i a_ix_i,w\right\rangle\right|
  \stackrel{\eqref{eq:witness-diagonal},\eqref{eq:r,w}}{\ge}
  \frac1r\sum_i|a_i|^2\|x_ip_i\|_1
  -\frac1r\left|
     \sum_{i\ne j}a_i\overline{a_j}
       \bigl\langle x_i,v_j^*\bigr\rangle
   \right|\\
  &{\ge}&
   \frac1r\sum_i|a_i|^2\|x_ip_i\|_1 -\frac1r
\left(\sum_{i\ne j}|a_i|^2|a_j|^2\right)^{1/2}
\left(\sum_{i\ne j}
|\langle x_i,v_j^*\rangle|^2\right)^{1/2}\\
  &\stackrel{\eqref{eq:right-mass},\eqref{eq:HS-off-diagonal},\eqref{eq:r,w}}{\ge}&
\frac{\delta}{2}\left(\sum_i|a_i|^2\right)^{1/2}.
\end{eqnarray*}
Hence, $(x_n)$ is $2$-co-lacunary.
\end{proof}

We now establish the $2$-co-lacunary subsequence theorem for  the weakly null subsequences  
with respect to arbitrary von Neumann algebras.

\begin{theorem}
\label{thm:haagerup-weak-null}
Let \(\M\) be an arbitrary von Neumann algebra.  Every seminormalized
weakly null sequence in \(L_1(\M)=\M_*\) has a \(2\)-co-lacunary
subsequence.
\end{theorem}

\begin{proof}It suffices to consider the case when $\M$ is $\sigma$-finite (see e.g. \cite[Lemma~4.3]{RandrianantoaninaHaagerup2002}). 
Set \(K=\{0,x_1,x_2,\ldots\}\). Then \(K\) is relatively weakly compact.

Applying \lemref{lem:four-corners} and passing to
a subsequence, we obtain pairwise disjoint projections
$(p_n)$ and the decomposition
\[
  x_n=a_n+b_n+c_n+d_n
\]
as in \eqref{eq:four-corners}.

Since the projections \(p_n\) are pairwise disjoint,
\(p_n\to0\) strongly, and by
\cite[Lemma 2.5(i)]{Takesaki1979}, \(p_n\to0\) ultraweakly,
equivalently, weak-star in \(\M\). Since \(K\) is relatively weakly compact, it follows from
\lemref{lem:weak-compact-compression} that
\begin{align}
 \|d_n\|_1\stackrel{\eqref{eq:four-corners}}{=}\|p_nx_np_n\|_1\le \sup_{x\in K}\|p_nxp_n\|_1
 \stackrel{\eqref{eq:typefree-diagonal-equi}}{\longrightarrow}0 \text{ as }n\to\infty
 \label{eq:d-vanishes}
\end{align}

If $\limsup_n\|b_n\|_1>0$, we may pass to
a further subsequence and relabel the corresponding pairs as
\((x_n,b_n)\) such that
\[
  \inf_n\|x_np_n\|_1
  \ge\inf_n\|b_n\|_1>0.
\]
By \lemref{lem:one-sided-witness}, \(\{x_n\}\) has a \(2\)-co-lacunary
subsequence.  Similarly, if
\(\limsup_n\|c_n\|_1>0\), \(\{x_n\}\) has a \(2\)-co-lacunary
subsequence also.

It remains to consider the case where
$\|b_n\|_1\to0$ and $\|c_n\|_1\to0$. Together with
\eqref{eq:d-vanishes}, we have that
\begin{align}\label{eq:x_n-a_n}
 \|x_n-a_n\|_1
  \le\|b_n\|_1+\|c_n\|_1+\|d_n\|_1
  \longrightarrow0 \text{ as }n\to \infty.
\end{align}
Passing to a subsequence if necessary, 
$(a_n)$ is seminormalized and weakly null.
By \lemref{lem:four-corners}, it is also
$1$-biequiintegrable.
By 
\corref{cor:biequi-weak-null}, \((a_n)\) has a \(2\)-co-lacunary subsequence
$(a_{n_k})$.  By \eqref{eq:x_n-a_n}, we may pass to a
further subsequence and relabel the corresponding pairs as
\((x_n,a_n)\) so that
\[
 \sum_n\|x_n-a_n\|_1<\infty.
\]
Then by \lemref{lem:co-lacunary}(1), a tail of \((x_n)\) is \(2\)-co-lacunary. This completes the proof.
\end{proof}

We now deduce the main theorem from
\thmref{thm:haagerup-weak-null}.
\begin{proof}[Proof of \thmref{thm:main}]
Using Rosenthal's \(\ell_1\)-theorem \cite{Rosenthal1974},  by the same argument as in the proof of~\cite[Corollary~3.7]{ParcetRandrianantoanina2006}, 
it suffices to consider a \rev{seminormalized weakly null sequence.}
Then, the conclusion follows from \thmref{thm:haagerup-weak-null}. 
\end{proof}

A natural approach to obtaining subsequence estimates
from martingale inequalities is to approximate a weakly
null sequence by a martingale difference sequence.
This approach works in the setting of hyperfinite semifinite von Neumann
algebras
\cite[Proof of Corollary~3.7]{ParcetRandrianantoanina2006}.
However, the following example shows that such an
approximation may fail for general finite von Neumann
algebras.

We first recall some terminology used in the following
example.

Let $\Gamma$ be a countable discrete group with identity
$e$, and let $\lambda$ be its left regular representation
on $\ell_2(\Gamma)$, defined by
\[
 \lambda(g)\delta_h=\delta_{gh}
 \qquad(g,h\in\Gamma).
\]
The associated group von Neumann algebra is
\[
 L(\Gamma)=\{\lambda(g):g\in\Gamma\}'',
\]
equipped with its canonical normalized trace
\[
 \tau(z)=\langle z\delta_e,\delta_e\rangle
 \qquad(z\in L(\Gamma)).
\]
In particular, $\tau(1)=1$.

Let $\Gamma$ be a group with identity $e$.
For $g\in\Gamma$, the conjugacy class of $g$ is
defined by
\[
 \operatorname{Conj}_{\Gamma}(g)
 =\{hgh^{-1}:h\in\Gamma\}.
\]
The group $\Gamma$ is called \emph{ICC}
(infinite conjugacy classes) if
\[
 \bigl|\operatorname{Conj}_{\Gamma}(g)\bigr|=\infty
 \qquad(g\in\Gamma\setminus\{e\}).
\]
In this case, $L(\Gamma)$ is a $\mathrm{II}_1$
factor \cite{GeShen2002}.
In particular,
let $\Gamma=\mathrm{SL}_3(\mathbb Z)$ be the
special linear group of degree three over
the integers, that is,
\[
 \mathrm{SL}_3(\mathbb Z)
 =\{g\in M_3(\mathbb Z):\det g=1\}.
\]
Then the group $\Gamma$ is countably infinite and ICC
\cite[Lemma~2]{GeShen2002}.
Consequently, $L(\Gamma)$ is a $\mathrm{II}_1$ factor.
Moreover, $L(\Gamma)$ has property~$(T)$
\cite{GeShen2002}.

A unitary representation of $\Gamma$ is said to be
\emph{weakly mixing} if it has no nonzero finite-dimensional
invariant subspaces. 
We write $W^*(a)$ for the von Neumann algebra
generated by $a$.

Finally, recall that $(\M,\tau)$ has \emph{property~$(T)$} if,
for every sequence of normal completely positive maps
$Q_n:\M\to\M$ satisfying
\[
 Q_n(1)\le1,\qquad \tau\circ Q_n\le\tau,
\]
and
\[
 \|Q_n(z)-z\|_2\longrightarrow0 \text{ as }n\to\infty
 \qquad(z\in\M).
\]
one has
\begin{equation}\label{eq:property-T-uniform}
  \sup_{\|z\|_\infty\le1}
 \|Q_n(z)-z\|_2\longrightarrow0 \text{ as }n\to\infty.
\end{equation}
Here $\|z\|_2=\tau(z^*z)^{1/2}$.

\begin{example}\label{ex:internal-martingale-obstruction}
Let $\Gamma=\mathrm{SL}_3(\mathbb Z)$ and let
$\M=L(\Gamma)$ be the associated group von Neumann
algebra, equipped with its canonical normalized trace $\tau$.
Then there exists a normalized weakly null sequence
$(x_n)$ in $L_1(\M,\tau)$ such that, for every
subsequence $(x_{n_k})$ and every martingale difference
sequence $(d_k)$ with respect to an increasing filtration
of von Neumann subalgebras of $\M$,
\begin{equation*}
 \limsup_{k\to\infty}
 \|x_{n_k}-d_k\|_1>0.
\end{equation*}
\end{example}

\begin{proof}
The algebra $\M$ is a $\mathrm{II}_1$ factor
with property~$(T)$ by \cite{GeShen2002}.
By \cite[Theorem~5]{GeShen2002}, there exist
self-adjoint elements $b,c\in\M$ such that
$W^*(b,c)=\M$.
Let
\[
 a=\frac{(b+ic)-\tau(b+ic)\cdot\one}{\bigl\|(b+ic)-\tau(b+ic)\cdot\one\bigr\|_1},
\]
it follows that
\[
 W^*(a)=\M,\qquad
 \tau(a)=0,\qquad
 \|a\|_1=1.
\]
Set 
\begin{align}\label{eq:a_infty}
H=\|a\|_\infty.
\end{align}
Since $\Gamma$ is ICC, the conjugation representation
of $\Gamma$ on $
L_2(\M)\ominus\mathbb C1
=
\{x\in L_2(\M):\tau(x)=0\}
$ is weakly mixing (see \cite[proof of Proposition~3.3]{AmrutamJiang2023}).
By 
\cite[Remark~3.2]{Popa2014}, there exist
$g_n\in\Gamma$ such that, with $u_n=\lambda(g_n)$,
\[
 x_n:=u_nau_n^*\rightharpoonup0
 \qquad\text{in }L_2(\M).
\]
Since $\M\subset L_2(\M)$ and $L_1(\M)^*=\M$,
the sequence $(x_n)$ is also weakly null in $L_1(\M)$.
Moreover,
\begin{align}\label{eq:x_n}
 \|x_n\|_1=1,\qquad \|x_n\|_\infty\stackrel{\eqref{eq:a_infty}}{=}H.
\end{align}

Assume by contradiction that there exist
a subsequence $(x_{n_k})$ and a martingale difference
sequence $(d_k)$ with respect to a filtration
$(\M_k)$ inside $\M$, such that
\begin{align}\label{eq:delta}
 \delta_k:=\|x_{n_k}-d_k\|_1\longrightarrow0 \text{ as }k\to\infty.
\end{align}
Let $\mathcal E_k:\M\to\M_k$ denote the
trace-preserving conditional expectations.

Since $\mathcal E_k(d_k)=d_k$, we have that
\begin{align*}
 \bigl\|x_{n_k}-\mathcal E_k(x_{n_k})\bigr\|_1
 =\bigl\|(x_{n_k}-d_k)
   -\mathcal E_k(x_{n_k}-d_k)\bigr\|_1
 \le\|x_{n_k}-d_k\|_1
   +\bigl\|\mathcal E_k(x_{n_k}-d_k)\bigr\|_1
 \overset{\eqref{eq:delta}}{\le}2\delta_k.
\end{align*}
Together with H\"older's inequality, we have that
\begin{align}
 \bigl\|x_{n_k}-\mathcal E_k(x_{n_k})\bigr\|_2^2
 \le
 \bigl\|x_{n_k}-\mathcal E_k(x_{n_k})\bigr\|_\infty
 \bigl\|x_{n_k}-\mathcal E_k(x_{n_k})\bigr\|_1
\overset{\eqref{eq:x_n}}{\le}4H\delta_k\longrightarrow0 \text{ as }k\to\infty.
 \label{eq:obstruction-L2}
\end{align}

Define
\[
 Q_k(z)=u_{n_k}^*
 \mathcal E_k\bigl(u_{n_k}zu_{n_k}^*\bigr)u_{n_k},
 \qquad z\in\M.
\]
Then $Q_k$ is the trace-preserving conditional
expectation onto
\[
 \mathcal N_k=u_{n_k}^*\M_ku_{n_k}.
\]
By \eqref{eq:obstruction-L2}, we have 
\begin{equation}\label{eq:obstruction-generator}
 \|Q_k(a)-a\|_2\longrightarrow0 \text{ as }k\to\infty.
\end{equation}

We claim that
\begin{equation}\label{eq:obstruction-pointwise}
 \|Q_k(z)-z\|_2\longrightarrow0
 \qquad(z\in\M).
\end{equation}
Indeed, let
\[
 \mathcal A=
 \bigl\{z\in\M:\|Q_k(z)-z\|_2\longrightarrow0\bigr\}.
\]
\rev{By \cite[Proposition~II.6.10.7]{Blackadar2006},
$Q_k$ extends to the orthogonal projection
from $L_2(\M)$ onto $L_2(\mathcal N_k)$.} Since $Q_k(x)Q_k(y)\in\mathcal N_k$,
for $x,y\in\mathcal A$ we have that
\begin{align*}
 \|xy-Q_k(xy)\|_2
 &\le
 \bigl\|xy-Q_k(x)Q_k(y)\bigr\|_2= \Bigl\|\bigl(xy-xQ_k(y)\bigr)+\bigl(xQ_k(y)-Q_k(x)Q_k(y)\bigr)\Bigr\|_2
 \\&\le \|x\|_\infty\|y-Q_k(y)\|_2
 +\|y\|_\infty\|x-Q_k(x)\|_2
\overset{\eqref{eq:obstruction-generator}}{\longrightarrow}0 \text{ as }k\to\infty.
\end{align*}
Together with  $Q_k(1)=1$ and
$Q_k(x^*)=Q_k(x)^*$, $\mathcal A$ is a unital $*$-subalgebra of $\M$.
For $z\in\overline{\mathcal A}^{\|\cdot\|_2}\cap\M$,
we have
\begin{align*}
0
&\le \limsup_{k\to\infty}\|z-Q_k(z)\|_2
\le \inf_{b\in\mathcal A}
 \limsup_{k\to\infty}
 \Bigl(\bigl\|(I-Q_k)(z-b)\bigr\|_2
       +\|b-Q_k(b)\|_2\Bigr)
\le \inf_{b\in\mathcal A}\|z-b\|_2=0,
\end{align*}
which implies that $\mathcal A$ is
relatively $L_2$-closed in $\M$.
Since $\mathcal A$ is a unital $*$-subalgebra,
the Kaplansky density theorem
\cite[Theorem~5.3.5]{KadisonRingrose1983}
implies that
\[
 W^*(\mathcal A)=\mathcal A.
\]
It follows from $a\in\mathcal A$ and $W^*(a)=\M$ that
$\mathcal A=\M$, proving
\eqref{eq:obstruction-pointwise}.

By \eqref{eq:obstruction-pointwise} and
property~$(T)$ of $\M$,
we obtain
\begin{equation}\label{eq:obstruction-uniform}
 \omega_k:=
 \sup_{\|z\|_\infty\le1}
 \|\mathcal E_k(z)-z\|_2
 =
 \sup_{\|z\|_\infty\le1}
 \|Q_k(z)-z\|_2
 \overset{\eqref{eq:property-T-uniform}}{\longrightarrow}0\text{ as }k\to\infty.
\end{equation}

Finally, since $\mathcal E_k(d_{k+1})=0$,
we obtain
\begin{eqnarray*}
 1\overset{\eqref{eq:x_n}}{=}\|x_{n_{k+1}}\|_1
 &\le&
 \bigl\|x_{n_{k+1}}-\mathcal E_k(x_{n_{k+1}})\bigr\|_1
 +\bigl\|\mathcal E_k(x_{n_{k+1}})\bigr\|_1\\&\overset{\eqref{eq:L1-Lp}}{\le}&\bigl\|x_{n_{k+1}}-\mathcal E_k(x_{n_{k+1}})\bigr\|_2
 +\bigl\|\mathcal E_k(x_{n_{k+1}}-d_{k+1})\bigr\|_1\\
 &\overset{\eqref{eq:x_n},\eqref{eq:obstruction-uniform}}{\le}&
 H\omega_k+
 \bigl\|\mathcal E_k(x_{n_{k+1}}-d_{k+1})\bigr\|_1
\overset{\eqref{eq:delta}}{\le}H\omega_k+\delta_{k+1}
 \longrightarrow0 \text{ as }k\to\infty,
\end{eqnarray*}
which is a contradiction to \eqref{eq:x_n}.
\end{proof}





\section{Applications of Theorems \ref{thm:main} and \ref{thm:transfer}}

\subsection{Compact operators, Dunford--Pettis operators, $\ell_2$-strictly singular operators}
We begin with a characterization of relatively norm compact subsets of
von Neumann algebra preduals and its consequence for compact operators.

\begin{proposition}\label{prop:norm-compactness}
Let \(\M\) be an arbitrary von Neumann algebra and let \(K\) be a bounded subset of
\(L_1(\M)\).  Then \(K\) is relatively norm compact if and only if it contains no
\(2\)-co-lacunary sequence.

Consequently, if \(K\) is relatively weakly compact but not relatively
norm compact, then it contains a weakly convergent \(2\)-co-lacunary sequence.
Moreover, for every Banach space \(X\) and every bounded operator
\(T:X\to L_1(\M)\),
\begin{align}
 T\text{ is compact}
 \quad\Longleftrightarrow\quad
 T(B_X)\text{ contains no \(2\)-co-lacunary sequence}.
 \label{eq:operator-compactness}
\end{align}
\end{proposition}

\begin{proof}
{\bf Necessity.} Every $2$-co-lacunary sequence $(x_n)$ with lower
constant $\gamma>0$ satisfies
\[
  \|x_n-x_m\|_1\ge\sqrt{2}\gamma
  \qquad(n\ne m).
\]
Thus a relatively norm compact set contains no
$2$-co-lacunary sequence.

{\bf Sufficiency.}
Conversely, if $K$ is
not relatively norm compact, \thmref{thm:main}
yields a $2$-co-lacunary sequence in $K$, which is a contradiction.


 If $K$ is relatively weakly compact but not
relatively norm compact, \thmref{thm:main} yields
a $2$-co-lacunary sequence in $K$.
By the Eberlein--\v Smulian theorem \cite[Chapter~III, p.~18]{Diestel1984}, there exists a further weakly convergent subsequence of the \(2\)-co-lacunary subsequence, which remains \(2\)-co-lacunary. This proves the second assertion.

Finally, since $T(B_X)$ is bounded, applying the
first assertion to $K=T(B_X)$ yields
\[
\begin{aligned}
T\text{ is compact}
&\iff T(B_X)\text{ is relatively norm compact}\\
&\iff T(B_X)\text{ contains no }
2\text{-co-lacunary sequence}.
\end{aligned}
\]
Hence, \eqref{eq:operator-compactness} holds.

\end{proof}




Recall that a bounded linear operator between Banach
spaces is called \emph{Dunford--Pettis}, or \emph{completely continuous},
if it sends weakly null
sequences to norm-null sequences (see \cite[p.~90]{DiestelUhl1977}).

We next apply Theorem \ref{thm:haagerup-weak-null} to
Dunford--Pettis operators. The following proposition
gives two characterizations for operators whose
domain or range is a von Neumann algebra predual.

\begin{proposition}\label{prop:DP-tests}
Let \(\M\) be an arbitrary von Neumann algebra.  Let \(X\) and
\(Y\) be Banach spaces.
\begin{enumerate}
\item A bounded linear operator \(T:X\to\M_*\) is Dunford--Pettis if and only if, for every
weakly null sequence \((x_n)\) in \(X\), the sequence \((Tx_n)\) has no
\(2\)-co-lacunary subsequence.
\item A bounded linear operator \(S:\M_*\to Y\) is Dunford--Pettis if and only if
\(\|Sx_n\|\to0\) for every weakly null \(2\)-co-lacunary sequence
\((x_n)\) in \(\M_*\).
\end{enumerate}
\end{proposition}

\begin{proof}
Since every \(2\)-co-lacunary sequence is bounded away from zero, the necessity of both assertions follows directly
from the definition of Dunford--Pettis operator.
Then we prove the sufficiency of the two implications.

If \(T\) is not
Dunford--Pettis, there exists a weakly null subsequence, relabeling as \((x_n)\), such that 
\[\inf_n\lVert Tx_n\rVert\ge\delta>0.\] 
Then, we may choose a seminormalized subsequence \((Tx_{n_i})\) of \((Tx_n)\). By \cite[Theorem~3.10]{Brezis2011}, \(T\) is weak-to-weak continuous. Hence,
\((Tx_{n_i})\) is a bounded weakly null sequence in \(\M_*\) and \thmref{thm:haagerup-weak-null} yields a \(2\)-co-lacunary subsequence of \((Tx_{n_i})\), also a  \(2\)-co-lacunary subsequence of \((Tx_{n})\).

Similarly, if $S$ is not Dunford--Pettis, choose a weakly null
sequence $(x_n)\subset\M_*$ such that $\|Sx_n\|\not\to0$.
Then, after passing to a subsequence and relabeling,
there exists a sequence $(x_n)$ such that
\begin{align}
 \inf_n\|Sx_n\|> 0.\label{eq:Sx_{n_i}}
\end{align}
Since $S$ is bounded, $(x_n)$ is seminormalized and
\thmref{thm:haagerup-weak-null} gives a
$2$-co-lacunary subsequence $(x_{n_k})$. 
Then it follows from \eqref{eq:Sx_{n_i}} that assertion (2) holds.
\end{proof}

Recall that \(T:X\to Y\) is \emph{\(\ell_2\)-strictly singular} if its
restriction to no subspace of \(X\) isomorphic to \(\ell_2\) is an
isomorphic embedding (see
\cite[pp.~61, 64]{PlichkoPopov1990}).

It was shown in 
\cite[Corollary~3.8]{HuangPlievSukochev2022}
 that every \(\ell_2\)-strictly singular
operator between certain hyperfinite von Neumann
algebra preduals is Dunford--Pettis.
We extend this result to arbitrary von Neumann
algebra preduals and obtain an equivalence between
the Dunford--Pettis property and
\(\ell_2\)-strict singularity
(\propref{prop:dp-l2ss-arbitrary}).

\begin{proposition}\label{prop:dp-l2ss-arbitrary}
Let \(\M\) and \(\N\) be arbitrary von Neumann algebras,
and let \(T:\M_*\to\N_*\) be a bounded linear operator.
Then
\[
  T\text{ is Dunford--Pettis}
  \quad\Longleftrightarrow\quad
  T\text{ is }\ell_2\text{-strictly singular}.
\]
\end{proposition}

\begin{proof}
 Every Dunford--Pettis operator is $\ell_2$-strictly singular.
Indeed, if $T$ is
bounded below on a subspace isomorphic to
$\ell_2$, the image of its unit vector basis
contradicts the Dunford--Pettis property.

Conversely, suppose that $T$ is not
Dunford--Pettis. There exist a bounded weakly
null sequence $(x_n)\subset L_1(\M)$ and
$\delta>0$ such that
\begin{align}\label{Sx_n ge delta}
  \|Tx_n\|_1\geq\delta
  \qquad(n\geq1).
\end{align}
It suffices to consider the case when $\M$ is $\sigma$-finite(see e.g. \cite[Lemma~4.3]{RandrianantoaninaHaagerup2002}). 


Fix 
\begin{align}\label{varrpsilon<}
0<\varepsilon<\delta/\bigl(4\|T\|\bigr).
\end{align}
Let \(D\in L_1(\mathcal M)_+\) be the Haagerup
density of a normal faithful state on
\(\mathcal M\). Then
\begin{align}\label{eq:D-normalization}
  s(D)=\one,\qquad
  \|D\|_1=\bigl\|D^{1/2}\bigr\|_2^2=1.
\end{align}
Since $\{0,x_n:n\ge1\}$ is weakly compact,
it follows from \cite[Theorem~4.14]{RaynaudXu2003} that $\{0,x_n:n\ge1\}$ is $1$-equiintegrable. Applying \lemref{lem:equi-approx}(2),
we obtain \(C_\varepsilon>0\) and contractions
\(a_n,b_n\in\mathcal M\) satisfying
\begin{align}
  \bigl\|x_n-C_\varepsilon(a_nD+Db_n)\bigr\|_1
  <\varepsilon
  \qquad(n\geq1).
  \label{eq:dp-sigma-additive-approximation}
\end{align}

Consider the Hilbert space
\[
  H=L_2(\M)\oplus_2L_2(\M)
\]
and the bounded linear operator
\[
  F:H\longrightarrow L_1(\mathcal M),
  \qquad
  F(\xi,\zeta)
  =C_\varepsilon
   \bigl(\xi D^{1/2}+D^{1/2}\zeta\bigr),\qquad \xi,\zeta\in L_2(\M).
\]
By the triangle inequality and the H\"older's inequality, we have that
\begin{eqnarray*}
\bigl\|F(\xi,\zeta)\bigr\|_1
&\leq& C_\varepsilon
\Bigl(\bigl\|\xi D^{1/2}\bigr\|_1+\bigl\|D^{1/2}\zeta\bigr\|_1\Bigr)\leq C_\varepsilon
\Bigl(\|\xi\|_2\bigl\|D^{1/2}\bigr\|_2
+\bigl\|D^{1/2}\bigr\|_2\|\zeta\|_2\Bigr)\\
&\stackrel{\eqref{eq:D-normalization}}{=}&C_\varepsilon\bigl(\|\xi\|_2+\|\zeta\|_2\bigr)\leq\sqrt2\,C_\varepsilon
\bigl(\|\xi\|_2^2+\|\zeta\|_2^2\bigr)^{1/2}=\sqrt2\,C_\varepsilon\bigl\|(\xi,\zeta)\bigr\|_H.
\end{eqnarray*}
Hence, we have that
\begin{align}\label{eq:F norm}
\|F\|_{H\to L_1(\mathcal{R})}\le \sqrt 2 C_\varepsilon,
\end{align}
where \(\|\cdot\|_{H\to L_1(\mathcal M)}\)
denotes the operator norm.
Set
\[
  q_n=\bigl(a_nD^{1/2},D^{1/2}b_n\bigr)\in H.
\]
It follows from \eqref{eq:dp-sigma-additive-approximation} that
\begin{align}\label{eq:x_n-q_n}
  \|x_n-Fq_n\|_1<\varepsilon.
\end{align}

Since \(a_n,b_n\) are contractions,
\[
  \|q_n\|_H\leq\sqrt2.
\]
Hence, the
Eberlein--\v Smulian theorem
\cite[Chapter~III]{Diestel1984}
allows us to pass to a subsequence such that
\[
  q_n\rightharpoonup q
  \qquad\text{in }H
\]
for some \(q\in H\).
Since \(x_n\rightharpoonup0\) and \(F\) is bounded,
we have that
\[x_n-Fq_n\rightharpoonup-Fq.\]
By weak lower semicontinuity of the norm,
\begin{align}
  \|Fq\|_1
  \leq\liminf_n\|x_n-Fq_n\|_1
  \stackrel{\eqref{eq:x_n-q_n}}{\le}\varepsilon.
  \label{eq:Fq-norm-bound}
\end{align}
Put
\begin{align}\label{eq:r_n,y_n}
  r_n=q_n-q,\qquad y_n=Fr_n.
\end{align}
Then \((r_n)\) and \((y_n)\) are bounded and
weakly null in \(H\) and \(L_1(\mathcal M)\),
respectively. Moreover,
\begin{align*} 
  \|x_n-y_n\|_1
  \leq\|x_n-Fq_n\|_1+\|Fq\|_1
  \stackrel{\eqref{eq:x_n-q_n},
  \eqref{eq:Fq-norm-bound}}{<}2\varepsilon,
\end{align*}
and then
\begin{align}\label{eq:dp-centered-image-bound}
  \|Ty_n\|_1
  \geq\|Tx_n\|_1-\|T\|\|x_n-y_n\|_1
  \stackrel{\eqref{Sx_n ge delta},\eqref{varrpsilon<}}{>}\frac{\delta}{2}.
\end{align}
Since \((r_n)\) is bounded and weakly null
in the Hilbert space \(H\), we may recursively
choose a subsequence \((r_{n_k})\) such that
\[
  \bigl|\langle r_{n_i},r_{n_j}\rangle_H\bigr|
  \leq2^{-i-j}
  \qquad(i<j).
\]
Writing \(B=\sup_n\|r_n\|_H\), for every finitely supported scalar sequence
\((c_k)\), we obtain that 
\begin{align*}
  \left\|\sum_k c_kr_{n_k}\right\|_H^2
  &\leq B^2\sum_k|c_k|^2
       +2\sum_{i<j}2^{-i-j}|c_i||c_j|
\\&\leq B^2\sum_k|c_k|^2
       +\left(\sum_{k\geq1}4^{-k}\right)
       \left(\sum_k|c_k|^2\right)
      \le
       \left(B^2+\frac13\right)
       \sum_k|c_k|^2.
\end{align*}
We have that
\[
\left\|\sum_k c_ky_{n_k}\right\|_1^2\stackrel{\eqref{eq:r_n,y_n}}{=}\left\|\sum_k c_kF(r_{n_k})\right\|_1^2\stackrel{\eqref{eq:F norm}}{\le}2 C_\varepsilon^2\left(B^2+\frac13\right)
       \sum_k|c_k|^2.
\]
After relabeling the subsequence,
we obtain \(C_0>0\) such that
\begin{align}
  \left\|\sum_n c_ny_n\right\|_1
  \leq C_0
  \left(\sum_n|c_n|^2\right)^{1/2}.
  \label{eq:dp-sigma-upper-l2}
\end{align}

By \eqref{eq:dp-centered-image-bound},
\((Ty_n)\) is a seminormalized weakly null
sequence in \(L_1(\N)\).
Applying \thmref{thm:haagerup-weak-null}
and passing to a further subsequence,
we obtain \(\gamma>0\) such that
\begin{align}
  \left\|\sum_n c_nTy_n\right\|_1
  \geq\gamma
  \left(\sum_n|c_n|^2\right)^{1/2}
  \label{eq:dp-sigma-lower-image}
\end{align}
for every finitely supported scalar sequence
\((c_n)\).
Combining \eqref{eq:dp-sigma-upper-l2} and
\eqref{eq:dp-sigma-lower-image}, we have
\[
  \frac{\gamma}{\|T\|}
  \left(\sum_n|c_n|^2\right)^{1/2}
  \leq
  \left\|\sum_n c_ny_n\right\|_1
  \leq
  C_0\left(\sum_n|c_n|^2\right)^{1/2}.
\]
Hence, $T$ is not $\ell_2$-strictly singular, completing the proof.
\end{proof}



We recall the notion of a spreading model
(see \cite[Definition~1.1]{OdellSchlumprecht1998}).
A sequence \((e_i)\) is called a \emph{spreading model}
generated by a sequence \((x_n)\) if, after
passing to a suitable subsequence of \((x_n)\),
for every \(m\geq1\) and all scalars
\(a_1,\ldots,a_m\),
\[
    \left\|\sum_{i=1}^m a_i e_i\right\|
    =
    \lim_{\substack{n_1<\cdots<n_m\\ n_1\to\infty}}
    \left\|\sum_{i=1}^m a_i x_{n_i}\right\|.
\]
We say that \((e_i)\) \emph{dominates} the unit vector
basis of \(\ell_2\) if there exists a constant
\(C>0\) such that
\[
    \left(\sum_{i=1}^m |a_i|^2\right)^{1/2}
    \leq
    C\left\|\sum_{i=1}^m a_i e_i\right\|
\]
for every \(m\geq1\) and all scalars
\(a_1,\ldots,a_m\)
(see \cite[Definition~3.1]{AndroulakisEtAl2005}).

We next derive some consequences of \thmref{thm:main}
for subspaces of von Neumann algebra preduals
and their spreading models.


\begin{corollary}
\label{cor:subspaces-spreading}
Let \(\M\) be an arbitrary von Neumann algebra.
Let \(X\) be isomorphic to a closed subspace of \(\M_*\).
\begin{enumerate}
\item Every such infinite-dimensional \(X\) contains a basic
\(2\)-co-lacunary sequence.
\item Every spreading model generated by a seminormalized weakly null sequence in
\(X\) dominates the unit vector basis of \(\ell_2\).
\end{enumerate}
\end{corollary}

\begin{proof}
  Let \(J:X\to\M_*\) be an isomorphic embedding.

  \begin{enumerate}
    \item By \cite[Chapter~V, Corollary~3]{Diestel1984},
$X$ contains a normalized basic sequence $(x_n)$.
Then, $(Jx_n)$ has no norm-convergent subsequence, also. The assertion follows directly from
\thmref{thm:main}, applied to
$(Jx_n)$.
    \item Since every subsequence
of a sequence generates the same spreading model, the assertion follows from \thmref{thm:haagerup-weak-null} directly.
\end{enumerate}
\end{proof}

\subsection{The $2$-colacunary property of noncommutative symmetric spaces}

\begin{corollary}\label{cor:Haagerup-Lp-dichotomy}
Let \(\M\) be an arbitrary von Neumann algebra
and let \(1\leq p\leq2\).
Every bounded sequence in the Haagerup space
\(L_p(\M)\) has either a norm-convergent
subsequence or a \(2\)-co-lacunary subsequence.
\end{corollary}
\begin{proof}
It suffices to consider the case when $\M$ is $\sigma$-finite (see e.g. \cite[Lemma~4.3]{RandrianantoaninaHaagerup2002}).

For $p=1$, the assertion follows from \thmref{thm:main}.
For $p=2$, let $(x_n)$ be bounded in $H$. Since the closed linear span of $(x_n)$ is a
separable Hilbert space, \(H\) embeds isomorphically into
$L_1(0,1)$ (see \cite[Chapter~1, Section~6, p.~8]{Kahane1985}).
Hence the conclusion follows from
\cite[Theorem~6]{AldousFremlin1982}.

Let $1<p<2$, and let $(x_n)$ be bounded in $L_p(\M)$. 
By \cite[Theorem~3.1]{Randrianantoanina2008} (see also \cite{HuangJungeSukochevZanin2025}),
$L_p(\M)$ embeds isomorphically into the predual of a
finite von Neumann algebra. Then, the conclusion follows from \thmref{thm:main}.
\end{proof}

\begin{remark}
The range \(1\leq p\leq2\) in
\corref{cor:Haagerup-Lp-dichotomy} is optimal.
The case when $p>2$ is obvious. 


For $0<p<1$, Henson et al.\
\cite[Theorem~4.2]{HensonEtAl1997}
constructed a bounded, uniformly separated sequence $(f_n)$
in $L_p(0,1)$ whose every subsequence has dense linear span.
Since $L_p(0,1)^*=\{0\}$
\cite[Section~1.47, p.~37]{Rudin1991},
no subsequence $(f_{n_k})$ can be $2$-co-lacunary. Indeed, otherwise, define 
\[
 \varphi\left(\sum_k c_kf_{n_k}\right):=c_1
\]
where $(c_k)$ is a finitely supported scalar sequence.
Since \((f_{n_k})\) is \(2\)-co-lacunary, it is linearly
independent, and hence \(\varphi\) is well defined.  
We observe that
the lower $\ell_2$-estimate would make $\left\|\sum_k c_kf_{n_k}\right\|\gtrsim \left\|(c_k)_k \right\| \ge   |c_1|$ so that the linear functional $\varphi$ is bounded on
$\operatorname{span}\{f_{n_k}:k\geq1\}$.
Since $\operatorname{span}\{f_{n_k}:k\geq1\}$ is dense in $L_p(0,1)$,
$\varphi$ would extend to a nonzero continuous linear functional
on $L_p(0,1)$, which is a contradiction.
\end{remark}

Sukochev and Zhou
\cite[Corollary~3.4]{SukochevZhou2020}
proved that every bounded sequence in
\((L_1+L_2)(\M,\tau)\) admits either a
norm-convergent subsequence or a
\(2\)-co-lacunary subsequence when \(\M\)
is hyperfinite and semifinite (acting on
a separable Hilbert space).
We extend this result to arbitrary
semifinite von Neumann algebras on a not necessarily separable Hilbert space.

\begin{corollary}\label{cor:L1-L2-dichotomy}
Let \(\M\) be a semifinite von Neumann algebra
equipped with a normal faithful semifinite trace
\(\tau\). Every bounded sequence in
\[
  (L_1+L_2)(\M,\tau)
  =L_1(\M,\tau)+L_2(\M,\tau)
\]
has either a norm-convergent subsequence or a
\(2\)-co-lacunary subsequence.
\end{corollary}

\begin{proof}
The proof is the same as that of \cite[Corollary 3.4]{SukochevZhou2020} by 
replacing \cite[Theorem 3.2]{SukochevZhou2020} and \cite[Lemma 38]{JSZ17} (in \cite[Corollary 3.3]{SukochevZhou2020}) with \thmref{thm:main} and 
\cite[Proposition A.4]{HJSZ26}
(see also \cite[Corollary~3]{HuangJungeSukochevZanin2025};  $(L_1+L_2)(\M,\tau)$ embeds into the predual of a \rev{finite} von Neumann algebra),  respectively. 
\end{proof}

Recall that a Banach function space
\(E=E(0,\infty)\) is called \emph{symmetric} if,
whenever \(f\in E\) and \(g\in L_0(0,\infty)\)
satisfy
\[
  \mu_t(g)\leq \mu_t(f)
  \qquad(t>0),
\]
we have \(g\in E\) and
\[
  \|g\|_E\leq\|f\|_E.
\]
Here \(\mu(f)\) denotes the decreasing
rearrangement of \(|f|\); see
\cite{BennettSharpley1988}.

For measurable functions \(f\) and \(g\), we
write \(f\prec\!\prec g\) if
\[
  \int_0^t\mu_s(f)\,ds
  \leq
  \int_0^t\mu_s(g)\,ds
  \qquad(t>0).
\]
A symmetric Banach function space \(E\) is called \emph{strongly symmetric} if
\[
  f,g\in E,\quad f\prec\!\prec g
  \quad\Longrightarrow\quad
  \|f\|_E\leq\|g\|_E.
\]
Thus, every fully symmetric space is strongly
symmetric; see \cite{DoddsDePagter2014}.

Moreover, \(E\) is said to have an
\emph{order-continuous norm} if, for every sequence
\((f_n)\subset E_+\) satisfying
\[
  f_n\downarrow0\quad\text{a.e.},
\]
one has
\[
  \|f_n\|_E\longrightarrow0;
\]
see \cite{BennettSharpley1988}.

We also recall that
\[
c_0=\bigl\{(a_n)_{n\geq1}:a_n\to0\bigr\}
\]
is the Banach space of scalar sequences
converging to zero, equipped with the norm
\[
  \bigl\|(a_n)\bigr\|_{c_0}
  =\sup_n|a_n|.
\]

Let \(\M\) be a semifinite von Neumann algebra
equipped with a normal faithful semifinite
trace \(\tau\), and let \(S(\M,\tau)\) denote
the algebra of all \(\tau\)-measurable
operators affiliated with \(\M\).
For \(x\in S(\M,\tau)\), let
\(\mu(x)\) denote its generalized singular
value function; see \cite{FackKosaki1986}.
The associated \emph{noncommutative symmetric space}
is defined by
\[
  E(\M,\tau)
  =
  \bigl\{x\in S(\M,\tau):\mu(x)\in E\bigr\},
\]
with norm
\[
  \|x\|_{E(\M,\tau)}
  =
  \bigl\|\mu(x)\bigr\|_E;
\]
see \cite{Kalton_Sukochev2008,LSZ}.

Using the preceding corollary, we further
extend the characterization established
in \cite[Theorem~1.1]{SukochevZhou2020}
to arbitrary semifinite von Neumann algebras,
removing  the assumptions of hyperfiniteness and 
separability.

\begin{theorem}\label{thm:semifinite-symmetric-dichotomy}
Let \(E=E(0,\infty)\) be an order-continuous
symmetric Banach function space such that
\[
  E\subset(L_1+L_2)(0,\infty).
\]
Let \(\M\) be an arbitrary semifinite
von Neumann algebra equipped with a normal
faithful semifinite trace \(\tau\).
Then the following assertions are equivalent:
\begin{enumerate}
\item Every bounded sequence in \(E(\M,\tau)\)
has either a norm-convergent subsequence or
a \(2\)-co-lacunary subsequence.
\item Every sequence of pairwise disjoint
elements in \(E(\M,\tau)\) either converges
to zero in norm or has a
\(2\)-co-lacunary subsequence.
\end{enumerate}
\end{theorem}

\begin{proof}The assertion follows from the same argument used in \cite[The proof of Theorem 1.1]{SukochevZhou2020}  by replacing \cite[Corollary 3.4]{SukochevZhou2020} with Corollary \ref{cor:L1-L2-dichotomy}. 
\end{proof}
\begin{remark}
Let \(E=E(0,\infty)\) be a symmetric Banach function space.
As shown by Sukochev and Zhou (see
\cite[pp.~2052--2053]{SukochevZhou2020}), if every bounded
sequence \((f_n)_{n\ge1}\subset E\) admits either a
norm-convergent subsequence or a \(2\)-co-lacunary
subsequence, then
\(
  E\subset (L_1+L_2)(0,\infty).
\)
\end{remark}

\subsection{Unconditional subsequence property}
A basic sequence $(u_k)$ is \emph{$K$-unconditional} if
\begin{equation*}
 \left\|\sum_k\gamma_ka_ku_k\right\|
 \le K\left\|\sum_ka_ku_k\right\|
\end{equation*}
for all finitely supported scalar sequences $(a_k)$ and all
signs $\gamma_k\in\{-1,1\}$. It is \emph{unconditional}
if it is $K$-unconditional for some $K<\infty$.

A Banach space $X$ has the \emph{unconditional subsequence
property} (\emph{USP}) if every seminormalized weakly null
sequence in $X$ has an unconditional basic subsequence.
It has the \emph{uniform unconditional subsequence property}
(\emph{uniform USP}) if there exists $K<\infty$ such that
every seminormalized weakly null sequence in $X$ has a
$K$-unconditional basic subsequence
\cite{OdellZheng2010}.

A Banach space $X$ is called a UMD space if,
for some $1<s<\infty$, there exists
$\beta_s(X)<\infty$ such that
\begin{equation}\label{eq:UMD}
 \left(\E\left\|
 \sum_{k=1}^m\xi_{k-1}d_k
 \right\|_X^s\right)^{1/s}
 \le \beta_s(X)
 \left(\E\left\|
 \sum_{k=1}^m d_k
 \right\|_X^s\right)^{1/s}
\end{equation}
for every finite $X$-valued $L_s$-martingale
difference sequence $(d_k)$ and every
predictable sign sequence $(\xi_{k-1})$
\cite{Burkholder1981,GirardiWeis2005}.
Recall that noncommutative $L
_p$-spaces are UMD when $1<p<\infty$ \cite[Corollary~7.7]{PisierXu2003} (see also \cite{Bourgain}, \cite{HS26} and \cite{BGM}).


\begin{corollary}\label{cor:umd}
Every UMD space $X$ has uniform USP.
\end{corollary}

\begin{proof}
By Lyapunov's inequality
\cite[p.~11, (1.1.17)]{BeningKorolev2002} and
\eqref{eq:UMD}, we have 
\begin{equation*}
 \E\left\|\sum_k\xi_{k-1}d_k\right\|_X
 \le\left(\E\left\|\sum_k\xi_{k-1}d_k\right\|_X^s\right)^{1/s}
\overset{\eqref{eq:UMD}}{\le}\beta_s(X)\left(\E\left\|\sum_kd_k\right\|_X^s\right)^{1/s}.
\end{equation*}
Then, the conclusion follows from Theorem~\ref{thm:transfer} with $Y=X$, $J=I_X$ and $r=1$.
\end{proof}

Oikhberg and Spinu proved USP for hyperfinite semifinite noncommutative
$L_p$-spaces, $1<p<\infty$,  in \cite[Corollary~2.3]{OikhbergSpinu2014}. They then asked
whether the hyperfiniteness assumption can be omitted.
The following corollary answers this question affirmatively.

\begin{corollary}\label{cor:nclp-usp}
Let $\M$ be any von Neumann algebra and let $1<p<\infty$. Then, $L_p(\M)$ has uniform USP. Moreover, constant for the uniform USP property depends only on $p$. 
\end{corollary}

\begin{proof}
By \cite[Corollary~7.7]{PisierXu2003}, $L_p(\M)$ is UMD. \rev{Note that the constant in \cite[Corollary~7.7]{PisierXu2003} depend only on $p$.}
Then the conclusion follows from
Corollary~\ref{cor:umd}.
\end{proof}

\subsection{Rosenthal-type inequalities for weakly null sequences in noncommutative $L_p$-spaces}
Following \cite[Section~2.2]{JiaoSukochevZaninZhou2017},
let $\mathcal N=\mathcal M\bar\otimes\ell_\infty$,
equipped with the tensor-product trace
$\widetilde\tau=\tau\otimes\mathrm{Tr}$,
where $\mathrm{Tr}$ is the counting trace on $\ell_\infty$.
For $1<p<\infty$, define
\begin{equation*}\label{eq:ZLp-definition}
 \|z\|_{Z_{L_p}^2(\mathcal N)}
 =
 \left(\int_0^1\mu_t(z)^p\,dt\right)^{1/p}
 +\|z\|_{L_1(\mathcal N)+L_2(\mathcal N)},
\end{equation*}
where
\[
 \|z\|_{L_1(\mathcal N)+L_2(\mathcal N)}
 =
 \inf_{z=a+b}
 \left(\|a\|_1+\|b\|_2\right).
\]
Here $(e_k)$ denotes the standard coordinate projections
of $\ell_\infty$.

The following result extends the Rosenthal-type inequalities of
Junge, Parcet and Xu \cite[Theorem~A]{JungeParcetXu2007}
to arbitrary seminormalized weakly null
sequences in noncommutative $L_p$-spaces.

\begin{theorem}
\label{thm:subsequence-JS}
Let $(\mathcal M,\tau)$ be a finite von Neumann algebra
equipped with a faithful normal normalized trace, and let
$1<p<\infty$.
Then every seminormalized weakly null sequence
$(x_n)\subset L_p(\mathcal M)$ has a basic subsequence
$(u_k)$ satisfying the following estimates.

\begin{enumerate}
\item If $1<p\le 2$, then
\begin{equation*}
 \left\|\sum_k a_k u_k\right\|_p
 \le C_p
 \left\|\sum_k a_k u_k\otimes e_k
 \right\|_{Z_{L_p}^2(\mathcal M\bar\otimes\ell_\infty)}.
\end{equation*}

\item If $2\le p<\infty$, then
\begin{equation*}
 \left\|\sum_k a_k u_k\otimes e_k
 \right\|_{Z_{L_p}^2(\mathcal M\bar\otimes\ell_\infty)}
 \le C_p
 \left\|\sum_k a_k u_k\right\|_p.
\end{equation*}
\end{enumerate}
Both estimates hold for every finitely supported complex
scalar sequence $(a_k)$, where $C_p$ depends only on $p$
and is independent of $\mathcal M$ and $(x_n)$.
\end{theorem}

\begin{proof}
By Corollary \ref{cor:nclp-usp},
we obtain a basic subsequence $(u_k)$ of $(x_n)$
and a constant $K_p$, depending only on $p$,
such that
\begin{equation}\label{eq:JS-reverse-unconditional}
 \left\|\sum_k a_k u_k\right\|_p
 \le K_p
 \left\|\sum_k\varepsilon_k a_k u_k\right\|_p.
\end{equation}
for all finitely supported complex scalars $(a_k)$
and all signs $\varepsilon_k\in\{-1,1\}$.


Fix $m\ge1$ and complex scalars
$a_1,\ldots,a_m$.
Let $\Omega=\{-1,1\}^m$ be equipped with
the probability measure
\[
 \mathbb P(\{\omega\})=2^{-m},
 \qquad \omega\in\Omega,
\]
and define $\varepsilon_k$ by setting 
\[
\varepsilon_k(\omega )= \varepsilon_k(\omega_1,\ldots,\omega_m)
 =\omega_k,\qquad 1\le k\le m.
\]
Then, $(\varepsilon_k)_{k=1}^m$ is a sequence
of independent symmetric signs.

Consider
\[
 \mathcal N=L_\infty(\Omega)\bar\otimes\mathcal M
\]
with the product trace $\mathbb P\otimes\tau$.
Set $\mathcal N_0=1\otimes\mathcal M$ and
\[
 \mathcal N_k
 =
 L_\infty\bigl(\sigma(\varepsilon_1,\ldots,
 \varepsilon_k)\bigr)\bar\otimes\mathcal M,
 \qquad 1\le k\le m.
\]
Since the signs $(\varepsilon_k)_{k=1}^m$ are independent and have mean zero,
\[
 d_k=\varepsilon_k\otimes(a_ku_k),
 \qquad 1\le k\le m,
\]
forms a martingale difference sequence
with respect to $(\mathcal N_k)_{k=0}^m$.

By the definition, we have 
\begin{align}\label{eq:d_k p}
 \left\|\sum_{k=1}^m d_k\right\|_{L_p(\mathcal N)}^p
 =
 \mathbb E\left\|\sum_{k=1}^m
 \varepsilon_ka_ku_k\right\|_p^p.
\end{align}
Moreover, since $|\varepsilon_k|=1$, we have
\begin{align}\label{eq:d_k e_k}
 \left|\sum_{k=1}^m d_k\otimes e_k\right|
 =
 1\otimes
 \left|\sum_{k=1}^m a_ku_k\otimes e_k\right|.
\end{align}


We now apply
\cite[Theorem~1.1]{JiaoSukochevZaninZhou2017}
to the martingale differences $(d_k)$.

If $1<p\le2$, then there exists $A_p$ depends only on $p$, such that
\begin{eqnarray*}
 \left\|\sum_{k=1}^m a_ku_k\right\|_p
 &\overset{\eqref{eq:JS-reverse-unconditional}}{\le}&
 K_p\left(
 \mathbb E\left\|\sum_{k=1}^m
 \varepsilon_ka_ku_k\right\|_p^p
 \right)^{1/p}
 \overset{\eqref{eq:d_k p}}{=}
 K_p\left\|\sum_{k=1}^m d_k
 \right\|_{L_p(\mathcal N)}
 \notag\\
 &\overset{\text{\cite[Theorem~1.1]{JiaoSukochevZaninZhou2017}}}{\le}& A_pK_p
 \left\|\sum_{k=1}^m d_k\otimes e_k
 \right\|_{Z_{L_p}^2(\mathcal N\bar\otimes\ell_\infty)}
 \overset{\eqref{eq:d_k e_k}}{=}
 A_pK_p
 \left\|\sum_{k=1}^m a_ku_k\otimes e_k
 \right\|_{Z_{L_p}^2(\mathcal M\bar\otimes\ell_\infty)}.
\end{eqnarray*}
If $2\le p<\infty$, 
then, by the same argument as above, there exists $A_p$ depends only on $p$, such that
\begin{eqnarray*}
 \left\|\sum_{k=1}^m a_ku_k\otimes e_k
 \right\|_{Z_{L_p}^2(\mathcal M\bar\otimes\ell_\infty)}
 &\overset{\eqref{eq:d_k e_k}}{=}&
 \left\|\sum_{k=1}^m d_k\otimes e_k
 \right\|_{Z_{L_p}^2(\mathcal N\bar\otimes\ell_\infty)}
 \overset{\text{\cite[Theorem~1.1]{JiaoSukochevZaninZhou2017}}}{\le} A_p
 \left\|\sum_{k=1}^m d_k\right\|_{L_p(\mathcal N)}
 \notag\\
 &\overset{\eqref{eq:d_k p}}{=}&
 A_p\left(
 \mathbb E\left\|\sum_{k=1}^m
 \varepsilon_ka_ku_k\right\|_p^p
 \right)^{1/p}
 \overset{\eqref{eq:JS-reverse-unconditional}}{\le}
 A_pK_p
 \left\|\sum_{k=1}^m a_ku_k\right\|_p.
\end{eqnarray*}
Taking $C_p=A_pK_p$ completes the proof.
\end{proof}

\subsection{Upper $\ell_p$ estimate}
A sequence $(u_k)$ satisfies an \emph{upper $\ell_p$ estimate} if
$\bigl\|\sum_ka_ku_k\bigr\|\le A\bigl(\sum_k|a_k|^p\bigr)^{1/p}$ for some $A$ and all
finitely supported scalar  $(a_k)$.

The following corollary answers
Haagerup, Rosenthal and Sukochev's problem
\cite[Section~5, Problem]{HaagerupRosenthalSukochev2003}
(which is equivalent to 
\cite[Section~5, Problem$'$]{HaagerupRosenthalSukochev2003} when $\mathcal N$ is a finite von Neumann algebra)
in the more general setting of  
arbitrary von Neumann algebras.


\begin{corollary}\label{cor:hrs}
Let $\M$ be a von Neumann algebra and let $1<p<2$.
There exists a constant $A_p<\infty$, depending
only on $p$, such that every seminormalized weakly
null sequence $(x_n)\subset L_p(\M)$ has a basic
subsequence $(u_k)$ of $(x_n)$ satisfying
\begin{equation*}
 \left\|\sum_k a_ku_k\right\|_p
 \le A_p\left(\sum_k|a_k|^p\|u_k\|_p^p\right)^{1/p}
\end{equation*}
for all finitely supported scalar sequences $(a_k)$.
\end{corollary}

\begin{proof}
For a seminormalized weakly null sequence
$(x_n)\subset L_p(\M)$, by
Corollary~\ref{cor:nclp-usp}, we can choose
an unconditional subsequence $(u_k)$
and a constant $K_p<\infty$ such that
\begin{equation}\label{eq:unconditional-Lp}
\left\|\sum_k\gamma_ka_ku_k\right\|_p
\le K_p\left\|\sum_ka_ku_k\right\|_p
\end{equation}
for all finitely supported scalar sequences
$(a_k)$ and all signs $\gamma_k\in\{-1,1\}$.

By the type $p$ of noncommutative $L_p$-space when $1< p<2$ ~\cite[Corollary~5.5]{PisierXu2003},
there exists $C>0$ such that
\begin{equation}\label{eq:Lp-rademacher}
\left(
\E_\varepsilon
\left\|\sum_k\varepsilon_ka_ku_k\right\|_p^p
\right)^{1/p}
\le C\left(\sum_k|a_k|^p\|u_k\|_p^p\right)^{1/p},
\end{equation}
where the $\varepsilon_k$'s are independent
symmetric signs.

Consequently,
\begin{eqnarray*}
\left\|\sum_k a_ku_k\right\|_p
&\overset{\eqref{eq:unconditional-Lp}}{\le}&
K_p\left(
\E_\varepsilon
\left\|\sum_k\varepsilon_ka_ku_k\right\|_p^p
\right)^{1/p}
\overset{\eqref{eq:Lp-rademacher}}{\le}
K_pC\left(\sum_k|a_k|^p\|u_k\|_p^p\right)^{1/p}
\nonumber\\
&\le&
K_pC\sup_n\|x_n\|_p
\left(\sum_k|a_k|^p\right)^{1/p}.
\end{eqnarray*}
This completes the proof.
\end{proof}

\begin{remark}
Note that the constants $C$ and $K_p$ depend only on $p$
(see the proof of \cite[Corollary~7.7]{PisierXu2003}
and the estimate in \cite[Corollary~5.5]{PisierXu2003}),
  independently
of $\M$. Consequently,  the constant  $A_p $
in Corollary~\ref{cor:hrs} can be chosen to depend only on $p$.
\end{remark}

\subsection{Weak-type $(1,1)$ subsequence estimates}
Let $X$ be a Banach space. 
Recall that a function $\Psi:X\to\mathbb R$ is
Lipschitz if there exists $L>0$ such that
\[
 \bigl|\Psi(x)-\Psi(y)\bigr|\le L\|x-y\|_X
 \qquad(x,y\in X).
\]

The following lemma is a nonhomogeneous version of
the tree argument used in the proof of
Theorem~\ref{thm:transfer}.

\begin{lemma}
\label{lem:additive-transfer}
Let $X$ be a Banach space and let
$\Psi:X\to[0,\infty)$ be a Lipschitz function
with $\Psi(0)=0$.
Suppose that, for some $C_\Psi\ge0$,
\begin{equation}\label{eq:additive-martingale-input}
 \E\Psi\left(\sum_{k=1}^m\xi_{k-1}d_k\right)
 \le C_\Psi\E\left\|\sum_{k=1}^m d_k\right\|_X
\end{equation}
for every finite $X$-valued martingale difference sequence
$(d_k)_{k=1}^m$ and every predictable sign sequence
$(\xi_{k-1})_{k=1}^m$.

Then, for every seminormalized weakly null sequence
$(x_n)\subset X$ and every $R,\eta>0$,
there exists a basic subsequence $(u_k)$ such that
\begin{equation}\label{eq:additive-transfer-conclusion}
 \Psi\left(\sum_{k=1}^m\gamma_ka_ku_k\right)
 \le C_\Psi\left\|\sum_{k=1}^m a_ku_k\right\|_X+\eta
\end{equation}
for all $m\ge1$, all complex scalars $a_1,\ldots,a_m$
satisfying $|a_k|\le R$, and all signs
$\gamma_k\in\{-1,1\}$.
\end{lemma}
\begin{proof}
The proof follows the same   argument (proof by contradiction) as that of
Theorem~\ref{thm:transfer}, with  \eqref{eq:bad-witness} replaced by
\begin{equation}\label{eq:additive-bad-witness}
 \bigl|b_k^F\bigr|\le R,\qquad
 \Psi\left(\sum_{k=1}^m\gamma_k^Fb_k^Fx_{n_k}\right)>C_\Psi\left\|\sum_{k=1}^m b_k^Fx_{n_k}\right\|_X+\eta.
\end{equation}
We sketch the proof below.

Since the coefficients in
\eqref{eq:additive-bad-witness} are uniformly bounded,
as in Steps~1--2 of the proof of
Theorem~\ref{thm:transfer}, the pruning and
Mazur's lemma yield a finite martingale difference
sequence $(d_k)$ and predictable signs $(\xi_{k-1})$.

In particular, \eqref{eq:rmd}, together with
\eqref{eq:stable-current} and
\eqref{eq:mean-bound}, gives
\begin{equation}\label{eq:additive-ordinary-error}
 \left\|\sum_{k=1}^N d_k(F)
 -\sum_{k=1}^m b_k^Fx_{n_k}\right\|_X
 \le H\sum_{k\ge1}\theta_k+\sum_{k\ge1}\eta_k.
\end{equation}
Similarly, by \eqref{eq:predictable-labels},
\begin{equation}\label{eq:additive-transformed-error}
 \left\|\sum_{k=1}^N\xi_{k-1}(F)d_k(F)
 -\sum_{k=1}^m\gamma_k^Fb_k^Fx_{n_k}\right\|_X
 \le H\sum_{k\ge1}\theta_k+\sum_{k\ge1}\eta_k
\end{equation}
for every terminal path $F=(n_1,\ldots,n_m)$,
where $H=\sup_n\|x_n\|_X$.

Let $L_\Psi$ be a Lipschitz constant of $\Psi$.
Then
\begin{align*}
C_\Psi\E\left\|\sum_{k=1}^m b_k^Fx_{n_k}\right\|_X
+\eta
&\overset{\eqref{eq:additive-bad-witness}}{<}
\E\Psi\left(\sum_{k=1}^m
\gamma_k^Fb_k^Fx_{n_k}\right)
\overset{\eqref{eq:additive-transformed-error}}{\le}
\E\Psi\left(\sum_{k=1}^N\xi_{k-1}d_k\right)
+L_\Psi\left(
H\sum_{k\ge1}\theta_k+\sum_{k\ge1}\eta_k
\right)\\
&\overset{\eqref{eq:additive-martingale-input}}{\le}
C_\Psi\E\left\|\sum_{k=1}^N d_k\right\|_X
+L_\Psi\left(
H\sum_{k\ge1}\theta_k+\sum_{k\ge1}\eta_k
\right)\\
&\overset{\eqref{eq:additive-ordinary-error}}{\le}
C_\Psi\E\left\|\sum_{k=1}^m b_k^Fx_{n_k}\right\|_X
+(C_\Psi+L_\Psi)
\left(
H\sum_{k\ge1}\theta_k+\sum_{k\ge1}\eta_k
\right).
\end{align*}
Since the sums of $(\theta_k)$ and $(\eta_k)$
can be chosen arbitrarily small and $\eta$ is fixed, the above equality yields a contradiction.
Hence, 
\rev{\eqref{eq:additive-transfer-conclusion} holds.}
\end{proof}

We recall two elementary properties of
distribution functions. Let $\M$ be a von Neumann
algebra equipped with a normal faithful semifinite
trace $\tau$. For $z\in L_1(\M,\tau)$, write
\begin{align}\label{eq:d_z}
d_z(t)=\tau\bigl(\one_{(t,\infty)}(|z|)\bigr),
 \qquad t>0.
\end{align}
By 
\cite[Proposition~2.2 and Lemma~2.5(v)]
{FackKosaki1986}, we have that
\begin{equation}\label{eq:distribution-sum}
 d_{z+w}(s+t)\le d_z(s)+d_w(t)
\end{equation}
for all $z,w\in L_1(\M,\tau)$ and $s,t>0$.
Moreover, since
\[
 t\,\one_{(t,\infty)}(|z|)\le |z|,
\]
it follows from \eqref{eq:d_z} that
\begin{equation}\label{eq:distribution-chebyshev}
 t\,d_z(t)\le\|z\|_1,
 \qquad t>0.
\end{equation}

We next apply Lemma~\ref{lem:additive-transfer} to obtain
a uniform weak-type \((1,1)\) estimate for sign changes
along a subsequence of every seminormalized weakly null sequence
in the noncommutative space \(L_1(\M,\tau)\), where \(\M\) is a
von Neumann algebra equipped with a normal  semifinite finite trace
\(\tau\).

\begin{corollary}\label{cor:weak-endpoint}
Let $\M$ be a von Neumann algebra equipped with a
normal faithful semifinite trace $\tau$.
There exists a universal constant $C_{\mathrm w}<\infty$
such that every seminormalized weakly null sequence
in $L_1(\M,\tau)$ has a basic subsequence $(u_k)$ satisfying
\[
 \left\|\sum_k\gamma_ka_ku_k\right\|_{1,\infty}
 \le C_{\mathrm w}\left\|\sum_ka_ku_k\right\|_1
\]
for all finitely supported complex scalar sequences
$(a_k)$ and all signs $\gamma_k\in\{-1,1\}$.
\end{corollary}

\begin{proof}
For $\lambda>0$, put
\[
 H_\lambda(z)=\int_\lambda^{2\lambda}d_z(t)\,dt,
 \qquad z\in L_1(\M,\tau),
\]
where $d_z$ is defined in \eqref{eq:d_z}.
By \cite[Theorem~3.9.14]{DPS}, we have that
\[
\begin{aligned}
 \bigl|H_\lambda(z)-H_\lambda(w)\bigr|
 \le\int_\lambda^{2\lambda}
 \bigl|d_z(t)-d_w(t)\bigr|\,dt
 \le\int_0^\infty
 \bigl|\mu_s(z)-\mu_s(w)\bigr|\,ds
 \le\|z-w\|_1,
\end{aligned}
\]
which implies that $H_\lambda$ is
$1$-Lipschitz. Moreover,
\begin{equation}\label{eq:weak-scale}
 \lambda d_z(2\lambda)\le H_\lambda(z).
\end{equation}

Applying Theorem~\ref{thm:MT} to the product algebra $\N=L_\infty(\Omega)\,\overline\otimes\,\M$
as in the proof of Corollary~\ref{cor:Lr}, \rev{we have that} 
\begin{align*}
 \E H_\lambda\left(\sum_k\xi_{k-1}d_k\right)
 &=\int_\lambda^{2\lambda}
   \E d_{\sum_k\xi_{k-1}d_k}(t)\,dt
\le
\left(\int_\lambda^{2\lambda}\frac{dt}{t}\right)
 \left\|\sum_k\xi_{k-1}d_k
 \right\|_{L_{1,\infty}(\N)}
 \\
 &\overset{\eqref{eq:MT}}{\le}
 C_{\mathrm{MT}}(\log2)\cdot\,
\E\left\|\sum_kd_k\right\|_1, 
 \nonumber
\end{align*}
for every finite $L_1(\M,\tau)$-valued martingale
difference sequence $(d_k)$ and every predictable sign sequence $(\xi_{k-1})$.


Pass to a basic subsequence $(x_n)$ with basis constant
\begin{equation}\label{eq:k_b}
 K_b\le2,
\end{equation}
and choose $B$
as in \eqref{eq:coefficient-bound}.

Applying \lemref{lem:additive-transfer} successively
with $R=B$ and $\eta=1$, first at scale $j=0$ and then
at scales $j=r,-r$ for each $r\ge1$, we obtain decreasing
infinite sets $(I_r)_{r\ge0}$ such that
\begin{equation}\label{eq:weak-fixed-scale}
 H_{2^{j-2}}\left(
 \sum_{n\in I_r}\gamma_na_nx_n\right)
 \le C_{\mathrm{MT}}(\log2)
 \left\|\sum_{n\in I_r}a_nx_n\right\|_1+1
 \qquad\bigl(|j|\le r\bigr)
\end{equation}
for all finitely supported coefficients $|a_n|\le B$
and all signs $\gamma$.

By \lemref{lem:weak-compact-compression} and
\lemref{lem:equi-approx}\textup{(2)}, for every
$\varepsilon>0$ there are a fixed $h\in L_1(\M,\tau)$
and contractions $a_n,b_n\in\M$ such that
\[
 \bigl\|x_n-(a_nh+hb_n)\bigr\|_1<\varepsilon.
\]
Then, it follows from \cite[Lemma~2.5(vi)]{FackKosaki1986} that
\begin{eqnarray}\label{eq:dx_n}
 \sup_n t\,d_{x_n}(t)
 &\overset{\eqref{eq:distribution-sum}}{\le}&
 \sup_n\left[
 t\,d_{a_nh}(t/4)+t\,d_{hb_n}(t/4)
 +t\,d_{x_n-a_nh-hb_n}(t/2)\right]\\\nonumber
 &\overset{\text{\cite[Lemma~2.5(vi)]{FackKosaki1986}}}{\le}&
 2t\,d_h(t/4)
 +\sup_n t\,d_{x_n-a_nh-hb_n}(t/2)\\
 \nonumber&\overset{\eqref{eq:distribution-chebyshev}}{\le}&
 2t\,d_h(t/4)
 +2\sup_n\|x_n-a_nh-hb_n\|_1\\\nonumber
 &\le &2t\,d_h(t/4)+2\varepsilon.
\end{eqnarray}

Since $d_h$ is decreasing, it follows that 
\[
 0\le2t\,d_h(t/4)
 \le16\int_{t/8}^{t/4}d_h(s)\,ds.
\]
As $d_h\in L_1(0,\infty)$, the right-hand side tends
to zero as $t\downarrow0$ and as $t\to\infty$.
Thus, letting $\varepsilon\downarrow0$, it follows from \eqref{eq:dx_n} that
\begin{equation}\label{eq:weak-two-endpoints}
 \sup_n t\,d_{x_n}(t)\longrightarrow0
 \quad\text{as }t\downarrow0
 \text{ and as }t\to\infty.
\end{equation}

By \eqref{eq:weak-two-endpoints}, we may choose strictly increasing positive integers $(j_k)$ such that
\begin{equation}\label{eq:weak-choice}
 2^j\sup_n
 d_{x_n}\left(\frac{2^{j-k-1}}B\right)
 =B2^{k+1}\left(\frac{2^{j-k-1}}{B}\right)\sup_n
 d_{x_n}\left(\frac{2^{j-k-1}}B\right)\le2^{-k}
 \qquad\bigl(j\in\mathbb Z,\ |j|>j_k\bigr).
\end{equation}
Choose $n_1<n_2<\cdots$ with $n_k\in I_{j_k}$,
and put $u_k=x_{n_k}$.

By homogeneity, it suffices to consider the case when 
\begin{equation}\label{eq:sum=1}
 \left\|\sum_{k=1}^m a_ku_k\right\|_1=1.
\end{equation}
Then \eqref{eq:coefficient-bound} gives
\begin{equation}\label{eq:|a_k|}
 |a_k|\le B.
\end{equation}

\rev{Fix $j\in\mathbb Z$.}
If $\{k:j_k<|j|\}=\varnothing$, then $2^j d_{\sum_{j_k<|j|}\gamma_ka_ku_k}
 \bigl(2^{j-1}\bigr)=0$.
Otherwise, since
\[
 \sum_{j_k<|j|}2^{j-k-1}\le2^{j-1},
\]
it follows from monotonicity of the distribution function that
\begin{eqnarray}
 2^j d_{\sum_{j_k<|j|}\gamma_ka_ku_k}
 \bigl(2^{j-1}\bigr)
 &{\le}&
 2^j d_{\sum_{j_k<|j|}\gamma_ka_ku_k}
 \left(\sum_{j_k<|j|}2^{j-k-1}\right)
 \nonumber
\overset{\eqref{eq:distribution-sum}}{\le}
 \sum_{j_k<|j|}2^j
 d_{\gamma_ka_ku_k}\bigl(2^{j-k-1}\bigr)
 \nonumber\\
 &\overset{\eqref{eq:|a_k|}}{\le}&
 \sum_{j_k<|j|}2^j
 d_{u_k}\left(\frac{2^{j-k-1}}B\right)
 \overset{\eqref{eq:weak-choice}}{\le}
 \sum_{k\ge1}2^{-k}=1.
 \label{eq:weak-head}
\end{eqnarray}
For $j_k\ge|j|$, we have
$n_k\in I_{j_k}\subseteq I_{|j|}$. Moreover, we have that
\[
 \left\|\sum_{j_k\ge|j|}a_ku_k\right\|_1
 \le (1+K_b)\left\|\sum_{k=1}^m a_ku_k\right\|_1
 \overset{\eqref{eq:k_b},\,\eqref{eq:sum=1}}{\le}3.
\]
Consequently, \eqref{eq:weak-scale} and
\eqref{eq:weak-fixed-scale} give
\begin{equation}\label{eq:weak-tail}
\begin{aligned}
 2^j d_{\sum_{j_k\ge|j|}\gamma_ka_ku_k}\bigl(2^{j-1}\bigr)\overset{\eqref{eq:weak-scale}}{\le}
 4H_{2^{j-2}}\left(
 \sum_{j_k\ge|j|}\gamma_ka_ku_k\right)
 \overset{\eqref{eq:weak-fixed-scale}}{\le}4\bigl(3C_{\mathrm{MT}}\log2+1\bigr).
\end{aligned}
\end{equation}
Combining \eqref{eq:distribution-sum}, \eqref{eq:weak-head} and 
\eqref{eq:weak-tail},
we obtain
\begin{align}\label{eq:sup C}
 \sup_{j\in\mathbb Z}
 2^j d_{\sum_{k=1}^m\gamma_ka_ku_k}\bigl(2^j\bigr)
 \le5+12C_{\mathrm{MT}}\log2.
\end{align}
By monotonicity of the distribution function,
\begin{eqnarray*}
\left\|\sum_{k=1}^m\gamma_ka_ku_k\right\|_{1,\infty}
&=&\sup_{\substack{j\in\mathbb Z,\ t>0\\
2^j\le t<2^{j+1}}}
t\,d_{\sum_{k=1}^m\gamma_ka_ku_k}(t)
\le2\sup_{j\in\mathbb Z}
2^jd_{\sum_{k=1}^m\gamma_ka_ku_k}\bigl(2^j\bigr)\\
&\overset{\eqref{eq:sup C}}{\le}&2\bigl(5+12C_{\mathrm{MT}}\log2\bigr)
=10+24C_{\mathrm{MT}}\log2.
\end{eqnarray*}
This completes the proof.
\end{proof}

{\bf 
AI statement}
At the outset of this work, the authors used OpenAI's GPT-6.0 Astra to investigate the \(2\)-co-lacunary subsequence problem for noncommutative \(L_1\)-spaces associated with semifinite von Neumann algebras. The model suggested an approach to the semifinite case. The authors extracted the underlying argument from that approach, formulated it as the transference principle in Theorem \ref{thm:transfer}, and applied the principle to the subsequent problems studied in this paper. The authors take responsibility for the mathematical statements and proofs in the manuscript.


\begin{thebibliography}{10}



\bibitem{AldousFremlin1982}
D.~J. Aldous and D.~H. Fremlin, \emph{Colacunary sequences in
  \(L\)-spaces}, Studia Math. \textbf{71} (1982), no.~3, 297--304.

\bibitem{AlspachJuddOdell2005}
D. Alspach, R. Judd, and E. Odell, \emph{The Szlenk index and local
  \(\ell_1\)-indices}, Positivity \textbf{9} (2005), no.~1, 1--44.

\bibitem{AmrutamJiang2023}
T. Amrutam and Y. Jiang,
\emph{On invariant von Neumann subalgebras rigidity property},
J. Funct. Anal.
\textbf{284} (2023), no.~5, 109804.

\bibitem{AndroulakisEtAl2005}
G. Androulakis, E. Odell, Th. Schlumprecht, and
  N. Tomczak-Jaegermann, \emph{On the structure of the spreading models of a
  Banach space}, Canad. J. Math. \textbf{57} (2005), no.~4, 673--707.

\bibitem{BeningKorolev2002}
V.~E. Bening and V.~Yu. Korolev,
\emph{Generalized Poisson Models and Their Applications
in Insurance and Finance},
VSP, Utrecht, 2002.

\bibitem{BennettSharpley1988}
C. Bennett and R. Sharpley,
\emph{Interpolation of operators},
Pure and Applied Mathematics, vol.~129,
Academic Press, Boston, MA, 1988.


\bibitem{BGM}
E. Berkson, T. A. Gillespie and P. S. Muhly, \emph{Abstract spectral
decompositions guaranteed by the Hilbert transform},
Proc. London Math. Soc. (3) \textbf{53} (1986), 489--517.

\bibitem{BessagaPelczynski1958}
C.~Bessaga and A.~Pe\l czy\'nski,
\emph{On bases and unconditional convergence of series in Banach spaces},
Studia Math. \textbf{17} (1958), 151--164.

\bibitem{Blackadar2006}
B. Blackadar,
\emph{Operator Algebras: Theory of $C^*$-Algebras
and von Neumann Algebras},
Encyclopaedia of Mathematical Sciences, vol.~122,
Springer, Berlin, 2006.

\bibitem{Bourgain1980}
J. Bourgain, \emph{A characterization of non-Dunford-Pettis operators on $L_1$}, Israel J. Math. \textbf{37} (1980), nos.~1--2, 48--53.

\bibitem{Bourgain}
J. Bourgain, \emph{Vector-valued singular integrals and the $H^1$--BMO
duality}, in \emph{Probability Theory and Harmonic Analysis}
(Cleveland, Ohio, 1983), J.-A. Chao and W. A. Woyczy\'nski (eds.),
Monogr. Textbooks Pure Appl. Math., vol.~98, Marcel Dekker,
New York, 1986, pp.~1--19.


\bibitem{Brezis2011}
H. Brezis, \emph{Functional analysis, Sobolev spaces and partial
  differential equations}, Universitext, Springer, New York, 2011.

\bibitem{Burkholder1981}
D.~L. Burkholder,
\emph{A geometrical characterization of Banach spaces
in which martingale difference sequences
are unconditional},
\emph{Ann. Probab.} \textbf{9} (1981),
997--1011.


\bibitem{CauseyDoebele2020}
R.~M. Causey and C. Doebele, \emph{A Ramsey theorem for pairs in
  trees}, Fund. Math. \textbf{248} (2020), no.~2, 147--193.


\bibitem{DiestelUhl1977}
J. Diestel and J.~J. Uhl, Jr.,
\emph{Vector measures},
Mathematical Surveys, vol.~15,
American Mathematical Society, Providence, RI, 1977.

\bibitem{Diestel1984}
J. Diestel, \emph{Sequences and series in Banach spaces}, Graduate Texts
  in Mathematics, vol.~92, Springer-Verlag, New York, 1984.

\bibitem{DiestelGraph2017}
R. Diestel, \emph{Graph theory}, 5th ed., Graduate Texts in Mathematics,
  vol.~173, Springer, Berlin, 2017.




\bibitem{DoddsDePagter2014}
P. G. Dodds and B. de Pagter,
\emph{Normed K\"othe spaces: a non-commutative viewpoint},
Indag. Math. (N.S.) \textbf{25} (2014),
no.~2, 206--249.

\bibitem{DPS}
P. Dodds, B. de Pagter, F. Sukochev, 
\emph{ Noncommutative integration and operator theory}, Progress in Mathematics, Birkh\"{a}user Cham, 2024.






\bibitem{FackKosaki1986}
T. Fack and H. Kosaki, \emph{Generalized \(s\)-numbers of
  \(\tau\)-measurable operators}, Pacific J. Math. \textbf{123} (1986), no.~2,
  269--300.
  
\bibitem{GeShen2002}
L. Ge and J. Shen,
\emph{Generator problem for certain property T factors},
Proc. Natl. Acad. Sci. USA \textbf{99} (2002),
no.~2, 565--567.

\bibitem{GirardiWeis2005}
M.~Girardi and L.~Weis,
\emph{Operator-valued martingale transforms
and $R$-boundedness},
\emph{Illinois J. Math.} \textbf{49} (2005),
487--516.


\bibitem{Haagerup1979}
U. Haagerup, \emph{\(L^p\)-spaces associated with an arbitrary von Neumann
  algebra}, Alg\`ebres d'op\'erateurs et leurs applications en physique
  math\'ematique (Proc. Colloq., Marseille, 1977), Colloq. Internat. CNRS, vol.
  274, CNRS, Paris, 1979, pp.~175--184.

\bibitem{HaagerupJungeXu2010}
U. Haagerup, M. Junge, and Q. Xu, \emph{A reduction method for
  noncommutative \(L_p\)-spaces and applications}, Trans. Amer. Math. Soc.
  \textbf{362} (2010), no.~4, 2125--2165.
  
\bibitem{HaagerupKraus1994}
U. Haagerup and J. Kraus, \emph{Approximation properties for group
  \(C^*\)-algebras and group von Neumann algebras}, Trans. Amer. Math.
  Soc. \textbf{344} (1994), no.~2, 667--699.

\bibitem{HaagerupRosenthalSukochev2003}
U. Haagerup, H.~P. Rosenthal, and F.~A. Sukochev, \emph{Banach
  embedding properties of non-commutative \(L^p\)-spaces}, Mem. Amer. Math.
  Soc. \textbf{163} (2003), no.~776, vi+68~pp.


\bibitem{HensonEtAl1997}
C. W. Henson, N. J. Kalton, N. T. Peck,
I. Tere\v{s}\v{c}\'ak, and P. Zlato\v{s},
\emph{Some Ramsey type theorems for normed
and quasinormed spaces},
Studia Math. \textbf{124} (1997), no.~1,
81--100.

\bibitem{HoudayerPopaVaes2013}
C. Houdayer, S. Popa and S. Vaes,
\emph{A class of groups for which every action is
$W^*$-superrigid},
Groups Geom. Dyn. \textbf{7} (2013),
no.~3, 577--590.

\bibitem{HuangJungeSukochevZanin2025}
J. Huang, M. Junge, F. Sukochev, and D. Zanin,
\emph{Embeddings of symmetric operator spaces into
\(L_p\)-spaces, \(1\leq p<2\), on finite von Neumann algebras},
Israel J. Math. \textbf{268} (2025), 73--116.


\bibitem{HJSZ26}
J. Huang, M. Junge, F. Sukochev, and D. Zanin,
{\it Isometric embeddings of noncommutative $L_p$-spaces into noncommutative symmetric spaces,
}
https://arxiv.org/pdf/2609.27827. 


\bibitem{HS26}
J. Huang, F. Sukochev, 
{\it The UMD property of symmetric operator spaces,}
https://arxiv.org/pdf/2609.14282.

\bibitem{HuangPlievSukochev2022}
J. Huang, M. Pliev, and F. Sukochev, \emph{\(\ell_2\)-strictly singular
  operators on the predual of a hyperfinite von Neumann algebra}, Proc.
  Amer. Math. Soc. \textbf{150} (2022), no.~4, 1623--1633.

\bibitem{JiaoSukochevZaninZhou2017}
Y. Jiao, F. Sukochev, D. Zanin and D. Zhou,
Johnson--Schechtman inequalities for noncommutative martingales,
\emph{J. Funct. Anal.} \textbf{272} (2017), no.~3, 976--1016.

\bibitem{JungeParcetXu2007}
M. Junge, J. Parcet and Q. Xu,
Rosenthal type inequalities for free chaos,
\emph{Ann. Probab.} \textbf{35} (2007), 1374--1437.

\bibitem{JSZ17}
M. Junge, F. Sukochev, and D. Zanin, \emph{Embeddings of operator ideals into $L_p$-spaces on finite von Neumann algebras}, Adv. Math. \textbf{312} (2017), 473-546. 





\bibitem{KadisonRingrose1983}
R. V. Kadison and J. R. Ringrose,
\emph{Fundamentals of the Theory of Operator Algebras},
Vol.~I: Elementary Theory,
Academic Press, New York, 1983.

\bibitem{Kahane1985}
J.-P. Kahane,
\emph{Some random series of functions},
2nd ed., Cambridge Studies in Advanced
Mathematics, vol.~5, Cambridge University
Press, Cambridge, 1985.

\bibitem{Kalton2005}
N.~J. Kalton, \emph{Rademacher series and decoupling}, New York J. Math.
  \textbf{11} (2005), 563--595.

\bibitem{Kalton_Sukochev2008}
N. Kalton, F. Sukochev, 
{\it Symmetric norms and spaces of operators}, 
J. Reine Angew. Math. \textbf{621} (2008), 81--121.

\bibitem{LSZ}
S. Lord, F. Sukochev, D. Zanin,
{\it Singular traces: Theory and applications,}
De Gruyter Studies in Mathematics, 46. De Gruyter, Berlin, 2013.

\bibitem{Nelson1974}
E. Nelson, \emph{Notes on non-commutative integration}, J. Funct. Anal.
  \textbf{15} (1974), 103--116.


\bibitem{OdellSchlumprecht1998}
E. Odell and Th. Schlumprecht,
\emph{A problem on spreading models},
J. Funct. Anal. \textbf{153} (1998), no.~1, 249--261.


\bibitem{OdellZheng2010}
E. Odell and B. Zheng,
\emph{On the unconditional subsequence property},
J. Funct. Anal. \textbf{258} (2010), no.~2, 604--615.


\bibitem{OikhbergSpinu2014}
T.~Oikhberg and E.~Spinu,
\emph{Operator ideals on non-commutative function spaces},
Integral Equations Operator Theory \textbf{79} (2014), no.~4, 507--532.



\bibitem{ParcetRandrianantoanina2006}
J. Parcet and N. Randrianantoanina, \emph{Gundy's decomposition for
  non-commutative martingales and applications}, Proc. Lond. Math. Soc. (3)
  \textbf{93} (2006), no.~1, 227--252.


\bibitem{Pequignot2017}
Y. Pequignot,
\emph{Towards better: A motivated introduction to
better-quasi-orders},
EMS Surv. Math. Sci. \textbf{4} (2017),
no.~2, 185--218.


\bibitem{PlichkoPopov1990}
A.~M. Plichko and M.~M. Popov,
\emph{Symmetric function spaces on atomless probability spaces},
Dissertationes Math. (Rozprawy Mat.) \textbf{306} (1990), 1--85.


\bibitem{PisierXu2003}
G. Pisier and Q. Xu, \emph{Non-commutative \(L^p\)-spaces}, in
  \emph{Handbook of the geometry of Banach spaces}, vol.~2, W.~B. Johnson
  and J. Lindenstrauss (eds.), North-Holland, Amsterdam, 2003,
  pp.~1459--1517.

\bibitem{Popa2014}
S. Popa,
A $\mathrm{II}_1$ factor approach to the
Kadison--Singer problem,
Comm. Math. Phys. \textbf{332} (2014), 379--414.

\bibitem{RabigerOrderBounded1991}
F.~R{\"a}biger,
\emph{Lower and upper \(2\)-estimates for order bounded sequences and
Dunford--Pettis operators between certain classes of Banach lattices},
in \emph{Functional Analysis (Austin, TX, 1987--1989)},
E.~Odell and H.~P. Rosenthal (eds.),
Lecture Notes in Math., vol.~1470, Springer-Verlag, Berlin, 1991,
pp.~159--170.

\bibitem{Randrianantoanina2008}
N. Randrianantoanina,
\emph{Embeddings of non-commutative \(L_p\)-spaces into
preduals of finite von Neumann algebras},
Israel J. Math. \textbf{163} (2008), 1--27.

\bibitem{RandrianantoaninaHaagerup2002}
N. Randrianantoanina, \emph{Kadec--Pe{\l}czy\'nski decomposition for
  Haagerup \(L^p\)-spaces}, Math. Proc. Cambridge Philos. Soc. \textbf{132}
  (2002), no.~1, 137--154.




\bibitem{RaynaudXu2003}
Y. Raynaud and Q. Xu, \emph{On subspaces of non-commutative
  \(L_p\)-spaces}, J. Funct. Anal. \textbf{203} (2003), no.~1, 149--196.

\bibitem{Rosenthal1974}
H.~P. Rosenthal, \emph{A characterization of Banach spaces containing
  \(\ell^1\)}, Proc. Natl. Acad. Sci. USA \textbf{71} (1974), no.~6,
  2411--2413.
\bibitem{Rosenthal1975}
H. Rosenthal, {\emph Convolution by a biased coin}, The Altgeld Book, University of Illinois, II,
1975/6.

\bibitem{Rudin1991}
W. Rudin,
\emph{Functional analysis},
2nd ed., International Series in Pure and Applied Mathematics,
McGraw-Hill, New York, 1991.

\bibitem{SukochevZhou2020}
F. Sukochev and D. Zhou, \emph{\(2\)-co-lacunary sequences in
  noncommutative symmetric Banach spaces}, Proc. Amer. Math. Soc.
  \textbf{148} (2020), no.~5, 2045--2058.

\bibitem{Takesaki1979}
M. Takesaki, \emph{Theory of operator algebras. I}, Springer-Verlag,
  New York--Heidelberg, 1979.


\end{thebibliography}
\end{document}